\documentclass[11pt,leqno]{article}
\usepackage{mathrsfs}
\usepackage{tabularx}
\usepackage{amsfonts}
\usepackage{pstricks}
\psset{unit=30pt}
\usepackage{amsmath, amsthm, amsfonts, amssymb, mathrsfs}
\newfam\msbfam
\font\tenmsb=msbm10    \textfont\msbfam=\tenmsb \font\sevenmsb=msbm7
\scriptfont\msbfam=\sevenmsb \font\fivemsb=msbm5
\scriptscriptfont\msbfam=\fivemsb
\def\Bbb{\fam\msbfam \tenmsb}
\newfam\bigfam
\font\tenbig=msbm10 scaled \magstep2   \textfont\bigfam=\tenbig
\font\sevenbig=msbm7 scaled \magstep2 \scriptfont\bigfam=\sevenbig
\font\fivebig=msbm5 scaled \magstep2
\scriptscriptfont\bigfam=\fivebig

\def\beqn{\begin{eqnarray}}
\def\eeqn{\end{eqnarray}}
\def\beqnn{\begin{eqnarray*}}
\def\eeqnn{\end{eqnarray*}}

\numberwithin{equation}{section}

\newtheorem{thm}{\hskip\parindent {Theorem}}[section]
\newtheorem{lem}{\hskip\parindent {Lemma}}[section]
\newtheorem{defn}{\hskip\parindent {Definition}}[section]
\newtheorem{cor}{\hskip\parindent {Corollary}}[section]
\newtheorem{rem}{\hspace\parindent{Remark}}[section]
\newtheorem{pro}{\hskip\parindent{Proposition}}[section]

\begin{document}
\baselineskip=15pt  \parskip=1pt

\begin{center}
\vspace{1.6cm}
{\Large \bf Riesz transforms associated with higher-order \\Schr\"odinger type operators }


\vspace{0.5cm}

{Qingquan\ Deng,\ Yong\ Ding,\ Xiaohua\ Yao }

{\linespread{1.1}
\renewcommand{\thefootnote}{}



\footnote{The first author is supported by NSFC
(No. 11301203) and the Fundamental Research Funds for the Central Universities
(CCNU-14A05037). The second author is supported by NSFC (No.11371057, No.11471033, No.11571160), SRFDP (No.20130003110003) and the Fundamental Research Funds for the Central Universities (No.2014KJJCA10). The third author is supported by NSFC (No. 11371158)
and the program for Changjiang Scholars and Innovative Research Team in University
(No. IRT13066).}}
\end{center}

\begin{center}{\bf Abstract}\\
\end{center}

In this paper, let $L=L_{0}+V$ be a Schr\"{o}dinger type operator where $L_{0}$ is higher order elliptic operator with complex coefficients in divergence form and $V$ is signed measurable function, under the strongly subcritical assumption on $V$, the authors study the $L^{q}$ boundedness of Riesz transforms $\nabla^{m}L^{-1/2}$ for $q\leq 2$ and obtain a sharp result.
Furthermore, the authors impose extra regularity assumptions  on $V$ to obtain the $L^{q}$ boundedness of Riesz transforms $\nabla^{m}L^{-1/2}$ for $q>2$. As an application, the main results can be applied to the operator $L=(-\Delta)^{m}-\gamma|x|^{-2m}$ for suitable $\gamma$.


 \tableofcontents

 \pagestyle{plain} 
 \renewcommand{\thepage}{\arabic{page}}
 \setcounter{page}{1} 
 \setcounter{equation}{0}

\section{Introduction}
Let $m\ge2,\ N\ge1$  and  $L_{0}$ be the following homogeneous elliptic operator of order $2m$ in divergence form
\begin{eqnarray}\label{divergence}
L_{0}f=\sum_{|\alpha|=|\beta|=m}\partial^{\alpha}(a_{\alpha,\beta}
\partial^{\beta}f)
\end{eqnarray}
which is interpreted via the following sesquilinear form
\begin{eqnarray}\label{Q_0 form}
\mathcal {Q}_{0}(f,g):=\sum_{|\alpha|=|\beta|=m}\int_{\mathbb{R}^{N}}a_{\alpha,\beta}(x)
\partial^{\beta}f(x)\overline{\partial^{\alpha}g(x)}dx, \ \ f,g\in W^{m,2}(\mathbb{R}^{N}),
\end{eqnarray}
where $W^{m,2}(\mathbb{R}^{N})$ is the Sobolev spaces and  for all $\alpha,\beta\in\mathbb{N}^{N}$ with $|\alpha|=|\beta|=m$,  $a_{\alpha,\beta}\in L^{\infty}(\mathbb{R}^{N},\mathbb{C})$ satisfying $a_{\alpha,\beta}=\overline{a_{\beta,\alpha}}$
and the strong  G{\aa}rding inequality
\begin{equation} \label{strong}
\begin{split}
\Lambda\|\nabla^{m}f\|^{2}_{L^{2}}\geq\sum_{|\alpha|=|\beta|=m}\int_{\mathbb{R}^{N}}a_{\alpha,\beta}(x)
\partial^{\beta}f(x)\overline{\partial^{\alpha}f(x)}dx\geq\rho\|\nabla^{m}f\|^{2}_{L^{2}},
\end{split}
\end{equation}
for some $\Lambda>\rho>0$ independent of $f\in
W^{m,2}(\mathbb{R}^{N})$. Assume that $V\in L^{1}_{\rm loc}(\mathbb{R}^{N})$ satisfies the strongly subcritical condition, that is,
 \begin{eqnarray}\label{condition}
\int_{\mathbb{R}^N}V_-|f|^2dx\le \mu\Big{(}\mathcal {Q}_{0}(f,f)+\int_{\mathbb{R}^N}V_+|f|^2dx \Big{)},
\end{eqnarray}
for some $\mu\in [0,1)$ and all $f\in W^{m,2}(\mathbb{R}^{N})$ with $\int_{\mathbb{R}^N}V_+|f|^2dx<\infty$, where
 $V_{\pm}(x):=\max\{0,\pm V(x)\}$ denote the positive and negative parts of $V$, respectively.
It follows from the KLMN theorem  (e.g. see Kato \cite[p.338]{Ka}) that $L$ is well-defined as a nonnegative self-adjoint operator associated to the  sesquilinear form
\begin{eqnarray}\label{form}
\mathcal{Q}(f,g):=\mathcal {Q}_{0}(f,g)+\int_{\mathbb{R}^{N}}V_+f\overline{g}dx-\int_{\mathbb{R}^{N}}V_-f\overline{g}dx
\end{eqnarray}
on the domain
$$D(\mathcal{Q})=\Big{\{}f\in W^{m,2}(\mathbb{R}^{N});\ \ \int_{\mathbb{R}^{N}}V_+|f|^2dx<\infty\Big{\}}.$$
For the sake of simplicity, we denote $\mathscr{A}_{m}(\rho,\Lambda,\mu)$ by the form sum $L=L_{0}+V$ associated to  $\mathcal{Q}$ defined as above and $\mathscr{A}_{m}=\bigcup_{\rho,\Lambda,\mu}\mathscr{A}_{m}(\rho,\Lambda,\mu)$ where $\mu\in[0,1)$ is understood the smallest constant such that (\ref{condition}) holds.

The famous Kato square root conjecture for second and higher elliptic operators $L$ with complex coefficients in divergence form has been proved by Auscher et al. (see \cite{A-H-M-T} \cite{ASLMT}), since then the $L^{q}$ theory for square root which consist in comparing $L^{1/2}f$ and $\nabla^{m} f$ ($m=1,2,\ldots,$) in $L^{p}$ norms have been extensively studied (see \cite{A2004} \cite{A2007} \cite{H-M03} and so on). There are two faces here, the
$L^{q}$ boundedness of Riesz transform, that is an inequality $\|\nabla^{m} f\|_{L^{q}}\leq C\|L^{1/2}f\|_{L^{q}}$, and its reverse $\|L^{1/2} f\|_{L^{q}}\leq C\|\nabla^{m}f\|_{L^{q}}$.

This paper is devoted to the study of $L^{q}$ boundedness of the Riesz transform associated to higher order Schr\"{o}dinger type operators $L=L_{0}+V\in\mathscr{A}_{m}$ where $L_{0}$ is a homogeneous elliptic operator of order $2m$ in divergence form.
In fact, the Riesz transform associated to $L\in \mathscr{A}_{m}$ with $V=0$ has been profoundly studied. Specifically, it follows from the classic Calder\'{o}n-Zygmund theory that the Riesz transform $\nabla^{m}L^{-1/2}_{0}$ ($m\geq1$) is bounded on $L^{q}(\mathbb{R}^{N})$ for all $1<q<\infty$ if $L_{0}$ is a  homogeneous elliptic operator of order $2m$  with smooth coefficients, in particular, $L_{0}=(-\Delta)^{m}$.
When $L_{0}$ is defined by $(\ref{divergence})$ with rough coefficients,
by using the off-diagonal estimate of $e^{-tL_{0}}$ and generalized Calder\'{o}n-Zygmund theory, the Riesz transform $\nabla^{m}L^{-1/2}_{0}$ is bounded on $L^{q}$ for all $q\in ((\frac{2N}{N+2m}-\varepsilon)\vee1, 2+\varepsilon')$ if $N\neq1$ and $q\in (1, \infty)$ if $N=1$ where $a\vee b=\max\{a,\ b\}$ and $\varepsilon,\varepsilon'$ are positive numbers depending on $L_{0}$ (see e.g. \cite{A2007} \cite{B-K04-1}).

When $V\neq0$, the study of Riesz transform $\nabla^{m}L^{-1/2}$ would be more complicated and may depends both $L_{0}$
 and $V$. Nevertheless, for the classic Schr\"{o}dinger operators $L=-\Delta+V$, there exist many interesting results on their associated Riesz transforms, which are essentially related to the properties of $e^{-tL}$. If $V>0$,
 based on the positivity and the Gaussian estimate of the kernel of $e^{-tL}$, the Riesz transform $\nabla L^{-1/2}$ has been extensively studied under some further assumptions on $V$ (see  \cite{A-A} \cite{D-O-Y} \cite{Ou} \cite{Sh95} \cite{Si} \cite{S-T} \cite{U-Z} \cite{Zh}).
 If $V$ is a signed potential satisfying strongly subcritical condition, notice that the kernel of $e^{-tL}$ may no longer positive and the Gaussian estimate may fails (see \cite{L-S-V}), Assaad \cite{As} and Assaad and Ouhabaz \cite{A-O} applied the $L^{p}$ theory and the off-diagonal estimates of $e^{-tL}$ to study the $L^{q}$ boundedness of $\nabla L^{-1/2}$. However, for the higher order Schr\"{o}dinger type operators $L=L_{0}+V$, we are aware that it is quite different  from classic Schr\"{o}dinger operators, even if $L_{0}=P(D)$ is a homogeneous elliptic operator of order $2m$ ($m\geq2$) with real constant coefficients and $V\geq0$. For
instance, the $e^{-tP(D)}(t\neq0)$ is not preserving-positivity
operator and also not a contractive one on $L^{p}(\mathbb{R}^{N})$
$(p\neq2)$ (see e.g.  Langer and Mazya \cite{L-M}), which makes
rather difficult to use famous Trotter formula of semigroup and
connection to probability. Moreover, these differences may also suggest that the study of Riesz transform $\nabla^{m}L^{-1/2}$ would be harder than $\nabla(-\Delta+V)^{-1/2}$.

Let $L=L_{0}+V\in\mathscr{A}_{m}$ be the non-negative self-adjoint operator associated to the  sesquilinear form $\mathcal{Q}$ in (\ref{form}), it is easy to see  $e^{-tL}$ is well-defined on $L^{2}(\mathbb{R}^{N})$ and
\begin{eqnarray}\label{Riesz}
\nabla^{m}L^{-1/2}f=c_{m}\int^{\infty}_{0}\sqrt{t}\nabla^{m}e^{-tL}f\frac{dt}{t},
\end{eqnarray}
which connect the Riesz transform $\nabla^{m}L^{-1/2}$ with semigroup $e^{-tL}$, hence the analysis of semigroup $e^{-tL}$  plays an important role in the study of Riesz transform.
In fact, as we mentioned before, the properties of semigroup $e^{-tL}$ such as the positivity, Gaussian estimates and off-diagonal estimates are all involved.
Here for $L\in \mathscr{A}_{m}$, it is well known that the kernel of $e^{-tL}$ is not likely to be positive and to have the Gaussian upper bound for $N>2m$.
Fortunately, we can obtain the off-diagonal estimates and $L^{p}$ theory for $e^{-tL}$ by observing a  invariant property of  $\mathscr{A}_{m}(\rho,\Lambda,\mu)$ (see Section \ref{sec2} below), which combined the generalized Calder\'{o}n-Zygmund theory to show the $L^{q}$ boundedness of  $\nabla^{m}L^{-1/2}$. We need to emphasize that the methods and results in studying the $L^{q}$ boundedness of $\nabla^{m}L^{-1/2}$ has twofold for $q\leq2$ and $q>2$, respectively, our paper hence will organized accordingly. First of all, for $q\leq2$, the results are addressed as follows:

\begin{thm}\label{thm4}
Assume that $m\geq2$, $\mu\in[0,1)$ and $L=L_{0}+V\in \mathscr{A}_{m}$ where $L_{0}$ is a homogeneous elliptic operator of order 2m in divergence form.

{\rm(i)}\ The Riesz transform $\nabla^{m}L^{-1/2}$ is bounded on $L^{q}(\mathbb{R}^{N})$ for all $q\in (\frac{2N}{N+2m}\vee1,2]$.

{\rm(ii)}\ The Riesz transform $\nabla^{m}L^{-1/2}$ is of weak type $(\frac{2N}{N+2m}\vee1, \frac{2N}{N+2m}\vee1)$ when $N\neq2m$.
\end{thm}
\vskip0.3cm

\begin{rem}\label{rem1-2}{\rm
(i)\ When $N\leq2m$, it is easy to see that the lower bound of the interval $(1,2]$ in Theorem \ref{thm4} is optimal. However, it is still unknown to us if the upper bound $2$ is sharp in the following sense: for given $q>2$, there exist  $L\in\mathscr{A}_{m}$ such that the Riesz transforms $\nabla^{m}L^{-1/2}$ is not bounded on $L^{q}(\mathbb{R}^{N})$.

(ii)\ When $N>2m\geq4$, the interval $(\frac{2N}{N+2m},2]$ in Theorem \ref{thm4} exists uniformly in $L\in \mathscr{A}_{m}$. Moreover, we will prove that $(\frac{2N}{N+2m},2]$ is sharp in the following sense: for given $q<\frac{2N}{N+2m}$ or $q>2$, there exist $L\in\mathscr{A}_{m}$ such that the Riesz transforms $\nabla^{m}L^{-1/2}$ is not bounded on $L^{q}(\mathbb{R}^{N})$, see Theorems \ref{sharp-1} below.


}
\end{rem}
\vskip 0.2cm

The $L^{q}$ boundedness of $\nabla^{m}L^{-1/2}$ for $q>2$ is subtle and dramatically different. We notice that if $N>2m$, the interval  $(\frac{2N}{N+2m},2]$ is optimal with respect to the class  $\mathscr{A}_{m}$ in Theorem \ref{thm4}. On the other hand, even for classic Schr\"{o}dinger operator $L=-\Delta+V$ with $V\geq0$, Shen's counterexample in \cite{Sh95} showed that extra regularity conditions on $V$ are necessary if one wants to the $L^q$ boundedness of $\nabla L^{-1/2}$ for $q>2$. In fact,  one can see \cite{A-A} and \cite{Sh95} for $V$ belonging to reverse H\"{o}lder class,  \cite{U-Z} for $V$ satisfying  Fefferman condition and \cite{As} for signed subcritical potential $V$ belonging to Kato class.

It is well known that the weak space $L^{\frac{N}{2},\infty}$ plays an important role in many studies of classic  Schr\"{o}dinger operators with critical potentials (see \cite{M-B,M-G}), where $L^{r,\infty}$ ($1\leq
r<\infty$) denote the weak $L^{r}(\mathbb{R}^{n})$ spaces, i.e.,
$$L^{r,\infty}=\big\{f:\ \|f\|_{L^{r,\infty}}=\sup_{\gamma>0}\gamma |\{x\in\mathbb{R}^{N};
|f(x)|>\gamma\}^{\frac{1}{r}}|<\infty\big\}.$$
A typical potential is the inverse square potential $V(x)=-c|x|^{-2}\in L^{\frac{n}{2},\infty}$, which is  widely studied in modern mathematical physics
and quantum mechanics (see for example \cite{B-P-S-T,RS2,R-S78,Sc,V-Z} and references therein). In remaining part of this paper, we are mainly devoted to the $L^q $ boundedness of  the Riesz transform $\nabla^m
L^{-1/2}$  for  $q>2$ with  potential $V$ belonging to $L^{\frac{N}{2m},\infty}(\mathbb{R}^{N})$.


\begin{thm}\label{thm5}
Let $m\geq2$, $\Lambda>\rho>0$, $\mu\in(0,1)$ and $N>2m$. Assume that $L=L_{0}+V\in \mathscr{A}_{m}(\rho,\Lambda,\mu)$ where $L_{0}$ is a homogeneous elliptic operator of order 2m in divergence form and $V\in L^{\frac{N}{2m},\infty}(\mathbb{R}^{N})$. Then the Riesz transform $\nabla^{m}L^{-1/2}$ is bounded on $L^{q}(\mathbb{R}^{N})$ for $q\in(\frac{(2+\varepsilon)N}{(1+\varepsilon)N+(2+\varepsilon)m},2+\varepsilon)$ where $\varepsilon$ is a positive constant.
\end{thm}

\begin{rem}\label{rem1-1}{\rm
(i)\  When $N>2m$, $V\in L^{\frac{N}{2m},\infty}(\mathbb{R}^{N})$ and $L_{0}$ is defined by (\ref{divergence}). By a suitable choice of $\|V\|_{L^{\frac{N}{2m},\infty}}$, we have that $L=L_{0}+V\in \mathscr{A}_{m}(\rho,\Lambda,\mu)$ for some $\Lambda>\rho>0$ and $\mu\in[0,1)$ automatically (see (ii) of Remark \ref{rem4-2}).

(ii)\ Let $N>2m\geq4$ and $L=(-\Delta)^{m}-\gamma|x|^{-2m}$. It follows from  Davies and Hinz \cite{D-H} that
\begin{eqnarray}\label{H-R}
\||x|^{-m}f\|^{2}_{L^{2}}\leq \kappa(m,N)\|\nabla^{m}f\|^{2}_{L^{2}}.
\end{eqnarray}
Thus if $0<\gamma<\frac{\rho}{\kappa(m,N)}$,  we have that  $L\in\mathscr{A}_{m}(\rho,\Lambda,\mu)$ with $0<\mu<\frac{\gamma \kappa(m,N)}{\rho}$ and $\nabla^{m}L^{-1/2}$ is bounded on $L^{q}(\mathbb{R}^{N})$ for all $q\in(\frac{(2+\varepsilon)N}{(1+\varepsilon)N+(2+\varepsilon)m},2+\varepsilon)$ where $\Lambda, \rho$ are constants appeared in (\ref{strong}) with $L_{0}=(-\Delta)^{m}$.
}
\end{rem}

Notice that by Theorem \ref{thm5}, we only obtain the constant $\varepsilon$ above is small and cannot be expressed explicitly for  $L=L_{0}+V\in \mathscr{A}_{m}(\rho,\Lambda,\mu)$ and
$V\in L^{\frac{N}{2m},\infty}(\mathbb{R}^{N})$, it is also of great interests to know how far we could push the upper bound. On the other hand, inspired by \cite{D-D-Y2} which deals with the classic Schr\"{o}dinger operator, it is possible to
 obtain the $L^q$ boundedness of $\nabla^m L^{-1/2}$ for $q$ lying in a larger and explicit interval if potential $V$ has extra regularity conditions. Therefore, we introduce such conditions for higher order Schr\"{o}dinger operators. Precisely, we assume
there exist a constant $q_{0}>2$ such that
\begin{eqnarray}
(A_{1}):\ \ \ \ \ \ \ \ \ \|\nabla^{m}L^{-1/2}_{0}\|_{L^{q_{0}}-L^{q_{0}}}<\infty,\nonumber
\end{eqnarray}
and
\begin{eqnarray}
(A_{2}):\ \ \ \ \ \ \ \ \ \|Vf\|_{L^{q_{0}}}\leq a\|L_{0}f\|_{L^{q_{0}}}+b\|f\|_{L^{q_{0}}},
\ \ \ f\in D_{q_{0}}(L_{0})\nonumber
\end{eqnarray}
with some constant $a,b>0$, where $L_{0}$ is defined by (\ref{divergence}) and $D_{q_{0}}(L_{0})$ is the domain of $L_{0}$ on $L^{q_{0}}$.
Then we present our result as follows.

\begin{thm} \label{ethm2-2-0}Let $N>2m$ and $L=L_{0}+V\in \mathscr{A}_{m}$ where $L_{0}$ is a homogeneous elliptic operator of order 2m in divergence form. Assume $(A_{1})$ and $(A_{2})$ hold for some $2<q_{0}<\frac{N}{m}$. Then there exist a constant $\delta_{q_{0}}>0$ depending on $q_{0}$ such that when
\begin{eqnarray}\label{eq1-1}
\|V\|_{L^{\frac{N}{2m},\infty}}\leq \delta_{q_{0}},
\end{eqnarray}
the Riesz transform $\nabla^{m}L^{-1/2}$ is bounded on $L^{q}(\mathbb{R}^{N})$ for all $\frac{Nq_{0}'}{N+mq_{0}'}<q<q_{0}$.
\end{thm}

\begin{rem}\label{rem1-3}
{\rm We give  several remarks about
the conditions $(A_{1})$ and $(A_{2})$.

(i)\ The condition $(A_{1})$ implies that $\nabla^{m}L^{-1/2}_{0}$ is bounded on $L^{q}(\mathbb{R}^{N})$ for all $\frac{Nq_{0}'}{N+mq_{0}'}<q<q_{0}$. When $N>2m$ and $2<q_{0}<\frac{2N}{N-2m}$, the condition $(A_{2})$ is equivalent to
\begin{eqnarray}
\Big\|V({\bf I}+L_{0})^{-1}\Big\|_{L^{q_{0}}-L^{q_{0}}}<\infty.\nonumber
\end{eqnarray}

(ii)\ When $L_{0}=P(D)$, the  homogeneous elliptic operator of  order $2m$ with constant real coefficients, the condition $(A_{1})$ holds for all $2<q_{0}<\infty$ and the condition $(A_{2})$ is equivalent to
\begin{eqnarray}
\Big\|V({\bf I}+L_{0}^{\frac{1}{\ell}})^{-\ell}\Big\|_{L^{q_{0}}-L^{q_{0}}}<\infty,\nonumber
\end{eqnarray}
for all $\ell\in \mathbb{N}$ and $2<q_{0}<\infty$. Moreover, if  $N>4m$ and $V\in L^{\frac{N}{2m},\infty}(\mathbb{R}^{N})$, by weak type H\"{o}lder inequality and the Sobolev embedding theorem, we have that
\begin{eqnarray}\label{eq1-3}
\|Vf\|_{L^{q_{0}}}\leq C\|V\|_{L^{\frac{N}{2m},\infty}}\|(-\Delta)^{m}f\|_{L^{q_{0}}}\leq C\|V\|_{L^{\frac{N}{2m},\infty}}\|L_{0}f\|_{L^{q_{0}}},\ \ \   1<q_{0}<\frac{N}{2m},\nonumber
\end{eqnarray}
which means that the condition  $(A_{2})$ holds for $2<q_{0}<\frac{N}{2m}$.}
\end{rem}

The paper is organized as follows: In Section 2, we develop the off-diagonal estimates for families of operators related to the heat semigroup $e^{-tL}$. Section 3 is devoted to the study of $L^{q}$ boundedness of Riesz transform for $q\leq 2$. In Section 4, we impose further regularities on potential $V$ to investigate the $L^{q}$ boundedness of Riesz transform for $q>2$. Finally, we apply our main results to some important Schr\"{o}dinger operators.

\section{Off-diagonal estimates}\label{sec2}
In this section, let $L\in\mathscr{A}_{m}$ $(m\geq2)$  be the nonnegative self-adjoint operator associated to the closed form $\mathcal{Q}$  defined by (\ref{form}). It is well-known that $L$ with domain $D(L)$  generates an analytic semigroup $e^{-tL}$ on $L^{2}(\mathbb{R}^{N})$.  Thus, the following families of operators
$$\{e^{-tL}\}_{t>0},\ \ \{tLe^{-tL}\}_{t>0} \ \  \text{and}\ \ \{\sqrt{t}\nabla^{m} e^{-tL}\}_{t>0}$$
are well-defined and uniformly bounded on $L^2(\mathbb{R}^N)$.
We  will focus on the off-diagonal estimates for those families of  operators, which have abundant applications to the study of $L^p$-extension of semigroup
$e^{-tL}$, $L^p$-spectra theory of $L$, the Hardy spaces and Riesz transform associated to $L$  and other related problems.

In subsection \ref{sec3}, we mainly establish the $L^2$ off-diagonal estimates for family of operators $\{e^{-tL}\}_{t>0}$.
If $V=0$, it is well known that $e^{-tL_{0}}$ satisfies the $L^2$ off-diagonal estimates (see e.g. \cite{A2007} \cite{B-K04-1} \cite{D95-3} \cite{D-D-Y1}). When $V\neq0$, by applying Davies' perturbation method in \cite{D95-3}, it is found that the off-diagonal estimates for $e^{-tL}$ contain a blow up term $e^{ct}$ $(c>0)$. Fortunately, we observe that for each fixed $m\geq2$, $\Lambda>\rho>0$ and $\mu\in[0,1)$, the space $\mathscr{A}_{m}(\rho,\Lambda,\mu)$ is invariance by dilations, which could help us to get rid of the blow up term $e^{ct}$. Consequently, the $L^2$ off-diagonal estimates for $\{tLe^{-tL}\}_{t>0}$ and $\{\sqrt{t}\nabla^{m} e^{-tL}\}_{t>0}$ will also be obtained.

In subsection \ref{sec2-2}, we will study the $L^{p}-L^2$ and $L^{2}-L^{p}$ off-diagonal estimates of  $\{e^{-tL}\}_{t>0}$ for appropriate $p$, which can be concluded by using the analyticity of $e^{-tL}$ and the same perturbation arguments as in subsection \ref{sec3}.

\subsection{The $L^{2}$ off-diagonal estimates for families of operators} \label{sec3}

We begin  with the following definition of $L^p$-$L^q$ off-diagonal estimates for a general family of operators.

\begin{defn}\label{def1}  \rm {{\it ($L^p$-$L^q$  off-diagonal estimates
for a family of operators}) \quad Let $\{S_{t}\}_{t>0}$ be a family
of uniformly bounded operators on $L^2(\mathbb{R}^{N})$. We say that $\{S_{t}\}_{t>0}$ satisfy the {\it $L^p$-$L^q$
off-diagonal estimates} for some $p,q\in [1,\infty)$ with $p\leq q$ if
there exist constants $C,c>0$ such that for all closed sets
$E,F\subset \mathbb{R}^{N}$, $t>0$ and $f\in
L^{2}(\mathbb{R}^{N})\cap L^{p}(\mathbb{R}^{N})$ supported in $F$,
the following estimate holds:
\begin{equation} \label{cond211}
\begin{split}\|S_{t}f\|_{L^{q}(E)}\leq
Ct^{\frac{N}{2m}(\frac{1}{q}-\frac{1}{p})}
e^{-c\big(\frac{d(E,F)}{t^{1/2m}}\big)^{2m/(2m-1)}} \|f\|_{L^{p}(F)},
\end{split}
 \end{equation}
where and in the sequel, $d(E,F)$ denotes the semi-distance induced
on sets by the Euclidean distance. In particular, if (\ref{cond211})
holds for $p=q$, then we say $\{S_{t}\}_{t>0}$ satisfy the {\it $L^{p}$
off-diagonal estimates}.}
\end{defn}

To study the $L^2$ off-diagonal estimates of the families of operators related to semigroup $e^{-tL}$, we first introduce the function space $\mathcal {E}_{m}(\mathbb{R}^N)$ which  consist of all bounded real-valued $C^{\infty}$ functions $\phi$ on $\mathbb{R}^{N}$ satisfying  $\|D^{\alpha}\phi\|_{L^{\infty}}\leq1$ for all $\alpha\in \mathbb{N}^N$ such that $1\leq|\alpha|\leq m$. It is easy to see that for $\lambda\in \mathbb{R}$ and $\phi\in \mathcal {E}_{m}(\mathbb{R}^N)$, the multiplication operators defined by  $e^{\pm\lambda\phi}$ are bounded and invertible on both $L^2(\mathbb{R}^N)$ and $W^{m,2}(\mathbb{R}^N)$. For $\lambda\in \mathbb{R}$ and $\phi\in \mathcal {E}_{m}(\mathbb{R}^N)$, consider the twist form
\begin{eqnarray}\label{perb}
&&\mathcal {Q}_{\lambda\phi}(f,g):=\mathcal {Q}(e^{-\lambda\phi}f,e^{\lambda\phi}g)\nonumber
\end{eqnarray}
with domain  $D(\mathcal {Q}_{\lambda\phi})=D(\mathcal {Q})$. Denote $L_{\lambda\phi}$ by
\begin{eqnarray}\label{perb-oper}
L_{\lambda\phi}f:=e^{\lambda\phi}Le^{-\lambda\phi}f,
\end{eqnarray}
where
$f\in D(L_{\lambda\phi})=\{f\in L^{2}(\mathbb{R}^{N}); \ e^{-\lambda\phi}f\in D(L) \}$. Then it follows that ${\rm Spec}(L_{\lambda\phi})={\rm Spec}(L)$ and $e^{-tL_{\lambda\phi}}f=e^{\lambda\phi}e^{-tL}e^{-\lambda\phi}f$ for $f\in L^2(\mathbb{R}^N)$.
Moreover,  we have that $D(H_{\lambda\phi})\subset D(\mathcal{Q})$ and
\begin{eqnarray}\label{Relation}
\mathcal{Q}_{\lambda\phi}(f,f)=\langle L_{\lambda\phi}f,f\rangle, \ \ \ \ \ \ \ \ \  \ f\in D(H_{\lambda\phi}).
\end{eqnarray}

\begin{lem}\label{le004}
Let $\lambda\in \mathbb{R}$, $\phi\in \mathcal {E}_{m}(\mathbb{R}^N)$, $L\in \mathscr{A}_{m}(\rho,\Lambda,\mu)$ and $L_{\lambda\phi}$ be the operator defined in (\ref{perb-oper}). Then there exist constants $C, c>0$ independent of specific $L$, $\lambda$ and $\phi$ such that
\begin{eqnarray}\label{eq001}
\|e^{-tL_{\lambda\phi}}f\|_{L^{2}(\mathbb{R}^N)}+\|tL_{\lambda\phi}e^{-tL_{\lambda\phi}}f\|_{L^{2}(\mathbb{R}^N)}\leq Ce^{c(1+\lambda^{2m})t}\|f\|_{L^{2}(\mathbb{R}^N)},\ \ \ t>0.
\end{eqnarray}
\end{lem}

\begin{proof}
The proof follows from similar procedures as the ones used in the proof of Lemmas 6 and 7 in Davies \cite{D95-3}, we omit the details here.
\end{proof}

\begin{rem}{\rm For fixed  $\Lambda>\rho>0$ and $\mu\in[0,1)$,  by the proof of  Lemma \ref{le004}, we know that the constant $c, C>0$ in (\ref{eq001}) are independent of $L\in \mathscr{A}_{m}(\rho,\Lambda,\mu)$, $\lambda\in\mathbb{R}$ and $\phi\in\mathcal {E}_{m}(\mathbb{R}^N)$, which will play a crucial role in the further study.}
\end{rem}
\vskip0.2cm
In order to deal with the off-diagonal estimates, we need to following  auxiliary functions
\begin{eqnarray}
\widetilde{d}(E,F):=\sup_{\phi\in \mathcal {E}_{m}(\mathbb{R}^{N})}\Big{[}\inf\{\phi(x)-\phi(y); \ x\in E,\ y\in F \}\Big{]}\nonumber
\end{eqnarray}
and
$\mathcal {A}(\phi):=\inf\{\phi(x);\ x\in E\}-\sup\{\phi(y);\ y\in F\},$
where $E$, $F$ are any two  subsets of $\mathbb{R}^{N}$.
It was shown in Davies \cite{D95-3} that
\begin{eqnarray}\label{equi-1}
d(E,F)\thicksim\widetilde{d}(E,F)\thicksim\sup\{\mathcal {A}(\phi);\ \phi\in \mathcal {E}_{m}\},
\end{eqnarray}
 where $E,F$ are disjoint compact convex subsets of $\mathbb{R}^N$, $d(E,F)=\{d(x,y);\ x\in E, y\in F\}$ is the deduced by the Euclidean distance and ``$A\thicksim B$" means that there exists some absolute constant $c$ only depending upon dimension $N$ such that $c^{-1}A\leq B\leq cA$. In fact, it is easy to show that (\ref{equi-1}) still holds for compact subsets of $\mathbb{R}^N$.
\vskip0.3cm

Now we show the $L^{2}$ off-diagonal estimates for the families of operators $\{e^{-tL}\}_{t>0}$ and  $\{tLe^{-tL}\}_{t>0}$ as follows:

\begin{thm}\label{thm1} Let $L=L_{0}+V\in \mathscr{A}_{m}$ where $L_{0}$ is a homogeneous elliptic operator of order 2m in divergence form. Then there exist constants $A,\ a>0$  such that  for arbitrary closed sets $E,F\subset\mathbb{R}^N$, $t>0$ and $f\in L^{2}(\mathbb{R}^N)$ with supp$f\subset F$ we have
\begin{eqnarray}\label{eq18}
\|e^{-tL}f\|_{L^{2}(E)}+\|tLe^{-tL}f\|_{L^{2}(E)}\leq A e^{-a\big{(}\frac{d(E,F)}{t^{1/2m}}\big{)}^{2m/(2m-1)}}\|f\|_{L^{2}(F)}.
\end{eqnarray}
\end{thm}

\begin{proof}  We only need to prove that
\begin{eqnarray}
\|tLe^{-tL}f\|_{L^{2}(E)}\leq A e^{-a\big{(}\frac{d(E,F)}{t^{1/2m}}\big{)}^{2m/(2m-1)}}\|f\|_{L^{2}(F)},\nonumber
\end{eqnarray}
since the proof for $\{e^{-tL}\}_{t>0}$ shares exactly the same procedure. Notice that for $L\in\mathscr{A}_{m}$, there exist $0<\rho<\Lambda<\infty$ and $\mu\in[0,1)$ such that $L\in \mathscr{A}_{m}(\rho,\Lambda,\mu)$. First, we show that there exist positive constants $C,\ a,\ c>0$ independent of $L\in \mathscr{A}_{m}(\rho,\Lambda,\mu)$ such that
\begin{eqnarray}\label{eq19}
\|tLe^{-tL}f\|_{L^{2}(E)}\leq C e^{ct}e^{-a\big{(}\frac{d(E,F)}{t^{1/2m}}\big{)}^{2m/(2m-1)}}\|f\|_{L^{2}(F)},
\end{eqnarray}
for arbitrary  compact set $E,F\subset\mathbb{R}^N$ and $f\in L^{2}(\mathbb{R}^N)$ with supp$f\subset F$.
The most important fact is that (\ref{eq19}) holds uniformly for any $L\in \mathscr{A}_{m}(\rho,\Lambda,\mu)$.

Denote $P_{E}$ and  $P_{E}$ the projections on $L^{2}(\mathbb{R}^{N})$ obtained by the characteristic function of the compact subset $E$ and $F$, respectively.  Then we have
\begin{eqnarray}\label{eq5}\|P_{E}tLe^{-tL}P_{F}\|_{L^{2}(\mathbb{R}^N)-L^{2}(\mathbb{R}^N)}
&\le&e^{-\lambda\mathcal {A}(\phi)}\|e^{\lambda\phi}tLe^{-tL}e^{-\lambda\phi}\|_{L^{2}(\mathbb{R}^N)-L^{2}(\mathbb{R}^N)}\nonumber\\
&=& e^{-\lambda\mathcal {A}(\phi)}\|tL_{\lambda\phi}e^{-tL_{\lambda\phi}}\|_{L^{2}(\mathbb{R}^N)-L^{2}(\mathbb{R}^N)},\nonumber
\end{eqnarray}
where $\mathcal{A}(\phi)=\inf\{\phi(x);\ x\in E\}-\sup\{\phi(y);\ y\in F\}$.  Thus by using Lemma \ref{le004} and (\ref{equi-1}) there exists constants $C,\ c,\ c'>0$ independent of $L\in \mathscr{A}_{m}(\rho,\Lambda,\mu)$, $\lambda\in\mathbb{R}$ and $\phi\in\mathcal {E}_{m}(\mathbb{R}^N)$ such that
\begin{eqnarray}\label{eq6}
\|P_{E}(tH)^{k}e^{-tH}P_{F}\|_{L^{2}(\mathbb{R}^N)-L^{2}(\mathbb{R}^N)}\leq C e^{-c'\lambda d(E,F)+c(1+\lambda^{2m})t}.\nonumber
\end{eqnarray}
Thus (\ref{eq19}) follows by choosing $\lambda=\big{(}\frac{c'd(E,F)}{2c t}\big{)}^{1/(2m-1)}$ and $a=\big{(}\frac{c'}{2c^{1/2m} }\big{)}^{2m/(2m-1)}$ in above inequality.

Now, we intend to get rid of the term $e^{c t}$ in the (\ref{eq19}) for any $t>1$. Let
$\mathcal {U}_{\delta}f=\delta^{\frac{N}{2}}f(\delta x)$ ($\delta>0$)
and $\mathcal Q$ be the associated form of $L\in\mathscr{A}_{m}(\rho,\Lambda,\mu)$ with domain $D(\mathcal Q)$, we define $$\mathcal {Q}^{\delta}(f,f):=\delta^{-2m}\mathcal Q(\mathcal {U}_{\delta}f,\mathcal {U}_{\delta}f)$$
with domain
$D(\mathcal {Q}^{\delta})=\{f\in W^{m,2}(\mathbb{R}^N):\ \ \mathcal {U}_{\delta}f\in D(\mathcal {Q})\}$. Then
\begin{eqnarray}\label{extra2-1}
\mathcal {Q}^{\delta}(f,f)
&=&\sum_{|\alpha|=|\beta|=m}\int_{\mathbb{R}^{N}}a_{\alpha,\beta}(\delta^{-1}x)
\partial^{\beta}f(x)\overline{\partial^{\alpha}g(x)}dx+\delta^{-2m}\int_{\mathbb{R}^{N}}V(\delta^{-1}x)f^{2}(x)dx\nonumber\\
&:=&\mathcal {Q}^{\delta}_{0}(f,f)+\langle V_\delta f,f\rangle,\nonumber
\end{eqnarray}
where $V_\delta=\delta^{-2m}V(\delta^{-1}\cdot)$. Moreover,
\begin{equation}
\begin{split}
\mathcal {Q}^{\delta}_{0}(f,f)=\sum_{|\alpha|=|\beta|=m}\int_{\mathbb{R}^{N}}a_{\alpha,\beta}(\delta^{-1}x)
\partial^{\beta}f(x)\overline{\partial^{\alpha}f(x)}dx\geq\rho\|\nabla^{m}f\|^{2}_{L^{2}(\mathbb{R}^{N})},\nonumber
\end{split}
\end{equation}
and
\begin{eqnarray}\label{extra2-2}
\int_{\mathbb{R}^{N}}V_{\delta,-}(x)f^{2}(x)dx
&\leq&\mu\Big(\mathcal Q^{\delta}_{0}(f,f)+\int_{\mathbb{R}^{N}}V_{\delta,+}(\delta^{-1}x)f^{2}(x)dx\Big),\nonumber
\end{eqnarray}
where $V_{\delta,\pm}$ be the positive and negative  parts of $V_{\delta}$ respectively. Therefore, we have that
\begin{eqnarray}\label{scaling-operator}
L_{\delta}=\delta^{-2m}\mathcal {U}^{-1}_{\delta}L\mathcal {U}_{\delta},\  \ \ \ \delta>0,
\end{eqnarray}
with domain $D(H_{\delta})=\{f\in L^{2}(\mathbb{R}^{N}):\ \mathcal {U}f\in D(H)\}$ is the nonnegative self-adjoint operator associated with $\mathcal {Q}^{\delta}$ and belongs to $\mathscr{A}_{m}(\rho,\Lambda,\mu)$.

On the other hand, notice that Lemma \ref{le004} holds uniformly for $L\in\mathscr{A}_{m}(\rho,\Lambda,\mu)$, which means that (\ref{eq19}) also holds for $L_{\delta}$.
Therefore, for all $0<\delta\leq1$, arbitrary  compact sets $E,F\subset\mathbb{R}^N$ and $f\in L^{2}(\mathbb{R}^N)$ with supp$f\subset F$, it follows from (\ref{eq19}) that
\begin{eqnarray}
\|tLe^{-tL}f\|_{L^{2}(E)}&=&\|\mathcal {U}_{\delta}(t\delta^{2m}L_{\delta})^{k}\mathcal {U}^{-1}_{\delta}e^{-t\delta^{2m}\mathcal {U}_{\delta}L_{\delta}\mathcal {U}^{-1}_{\delta}}f\|_{L^{2}(E)}\nonumber\\
&=&\|(t\delta^{2m}H_{\delta})^{k}e^{-t\delta^{2m}H_{\delta}}\mathcal {U}^{-1}_{\delta}f\|_{L^{2}(\delta E)}\nonumber\\
&\leq& Ce^{ct\delta^{2m}}e^{-a\big{(}\frac{d(\delta E,\delta F)}{\delta t^{1/2m}}\big{)}^{2m/(2m-1)}}\|f\|_{L^{2}(F)}\nonumber
\end{eqnarray}
where  $\delta E:=\{\delta x; \ x\in E\}$ for any set $E\subset \mathbf{}\mathbb{R}^N$ and the constants $C,\ a, \ c$ are independent of $\delta$ due to the fact that the estimate (\ref{eq19}) holds uniformly in $L_{\delta}\in \mathscr{A}_{m}(\rho,\Lambda,\mu)$.
Note that  $d(\delta E,\delta F)=\delta d(E,F)$, then by choosing $\delta=t^{-1/2m}$ $(t>1)$ we obtain
\begin{eqnarray}\label{off-diagonal}
\|tLe^{-tL}f\|_{L^{2}(E)}
&\leq& C e^{ct\delta^{2m}}e^{-a\big{(}\frac{\delta d(E,F)}{\delta t^{1/2m}}\big{)}^{2m/(2m-1)}}\|f\|_{L^{2}(F)}\nonumber\\
&\leq& Ae^{-a\big{(}\frac{d(E,F)}{t^{1/2m}}\big{)}^{2m/(2m-1)}}\|f\|_{L^{2}(F)}.
\end{eqnarray}

Finally, observe that for arbitrary closed set $E,F\subset \mathbb{R}^{N}$, since $E=\cup^{\infty}_{\ell=1}E_{\ell}$ and $F=\cup^{\infty}_{\ell=1}E_{\ell}$ where both $\{E_{\ell}\}^{\infty}_{\ell=1}$ and $\{F_{\ell}\}^{\infty}_{\ell=1}$ are increasing monotone sets sequences. Then by a limitation procedure, it is easy to see that (\ref{off-diagonal}) holds for arbitrary closed set. Thus we finish the proof of Theorem \ref{thm1}.
\end{proof}
\vskip0.2cm

Now we turn to study the $L^2$ off-diagonal estimates of $\{\sqrt{t}\nabla^{m}e^{-tL}\}_{t>0}$.

\begin{thm}\label{pro20}
Let $L=L_{0}+V\in \mathscr{A}_{m}$ where $L_{0}$ is a homogeneous elliptic operator of order 2m in divergence form. Then there exist constants $B, b>0$ such that for arbitrary closed set $E,F\subset\mathbb{R}^N$, $t>0$ and $f\in L^{2}(\mathbb{R}^N)$ with supp$f\subset F$,
\begin{eqnarray}
\|\sqrt{t}\nabla^{m}e^{-tL}f\|_{L^{2}(E)}\leq B e^{-b\big{(}\frac{d(E,F)}{t^{1/2m}}\big{)}^{2m/(2m-1)}}\|f\|_{L^{2}(F)}.\nonumber
\end{eqnarray}
\end{thm}

\begin{proof} By the proof of Theorem \ref{thm1}, we only need to  show that there exist constants $C,c>0$ independent of the potential $L\in \mathscr{A}_{m}(\rho,\Lambda,\mu)$, $\lambda\in \mathbb{R}$ and $\phi\in \mathcal {E}_{m}(\mathbb{R}^N)$ such that
\begin{eqnarray}\label{eq30}
\|e^{\lambda\phi}\sqrt{t}\nabla^{m}e^{-tL}e^{-\lambda\phi}f\|_{L^{2}(\mathbb{R}^N)}\leq C e^{c(1+\lambda^{2m})t}\|f\|_{L^{2}(\mathbb{R}^N)}
\end{eqnarray}
for all $t>0$.

Once (\ref{eq30}) holds, for arbitrary compact set $E,F\subset\mathbb{R}^N$ and $f\in L^{2}(\mathbb{R}^N)$ with supp$f\subset F$, the same procedure as in the proof of Theorem \ref{thm1} can lead to
\begin{eqnarray}\label{eq30'}
\|\sqrt{t}\nabla^{m}e^{-tH}f\|_{L^{2}(E)}\leq Ce^{c t} e^{-b\big{(}\frac{d(E,F)}{t^{1/2m}}\big{)}^{2m/(2m-1)}}\|f\|_{L^{2}(F)},
\end{eqnarray}
where the  constants $C, c, b>0$ are independent of the  $L\in \mathscr{A}_{m}(\rho,\Lambda,\mu)$.

To remove the term $e^{c t}$ for $t>1$. Similarly to the proof of Theorem \ref{thm1}, let  $L_{\delta}$ be the  operators defined by (\ref{scaling-operator}), it is easy to see that (\ref{eq30'}) uniformly holds for $L_{\delta}$ ($\delta>0$).
Therefore,
\begin{eqnarray}
\|\sqrt{t}\nabla^{m}e^{-tL}f\|_{L^{2}(E)}&=&\|\sqrt{t}\nabla^{m}e^{-t\delta^{2m}\mathcal {U}_{\delta}L_{\delta}\mathcal {U}^{-1}_{\delta}}f\|_{L^{2}(E)}\nonumber\\
&=&\|\mathcal {U}_{\delta}\sqrt{(\delta^{2m}t)}\nabla^{m}e^{-t\delta^{2m}L_{\delta}}\mathcal {U}^{-1}_{\delta}f\|_{L^{2}(E)}\nonumber\\
&\leq& Ce^{ct\delta^{2m}}e^{-b\big{(}\frac{d(\delta E,\delta F)}{\delta t^{1/2m}}\big{)}^{2m/(2m-1)}}\|f\|_{L^{2}(F)}.\nonumber
\end{eqnarray}
Then by taking $\delta=t^{-1/2m}$, we can obtain that
\begin{eqnarray}
\|\sqrt{t}\nabla^{m}e^{-tL}f\|_{L^{2}(E)}
&\leq& Ce^{c t\delta^{2m}}e^{-b\big{(}\frac{\delta d(E,F)}{\delta t^{1/2m}}\big{)}^{2m/(2m-1)}}\|f\|_{L^{2}(F)}\nonumber\\
&\leq& Be^{-b\big{(}\frac{d(E,F)}{t^{1/2m}}\big{)}^{2m/(2m-1)}}\|f\|_{L^{2}(F)}.\nonumber
\end{eqnarray}

It remains to prove (\ref{eq30}). In fact, (\ref{eq30}) can be easily obtained by using the same procedure as in the proof of \cite[Proposition 3.1]{B-K04-1}. Hence, we finish the proof.
\end{proof}

\subsection{The $L^{p}-L^{q}$ off-diagonal estimates for families of operators}\label{sec2-2}

This subsection is devoted to establishing $L^{p}$-$L^{2}$ ($L^{2}$-$L^{p}$) off-diagonal estimates for the following families of operators
$$\{e^{-tL}\}_{t>0},\ \ \{tLe^{-tL}\}_{t>0}\ \ \text{and}\ \
  \{\sqrt{t}\nabla^{m} e^{tL}\}_{t>0},$$
which is the bridge between $L^{p}-L^{q}$ estimates and uniformly boundedness.
 We first introduce the  $L^{p}$-$L^{q}$ estimates for a general family operators.

\begin{defn}\label{def2}\rm {({\it $L^p$-$L^q$  estimates for a family of
operators})\quad Let $\{S_{t}\}_{t>0}$ be a family of uniformly bounded operators on $L^{2}(\mathbb{R}^{N})$. We
say that $\{S_{t}\}_{t>0}$ satisfy the \emph{ $L^p$-$L^q$ estimates} for some
$p,q\in [1,\infty)$ with $p\leq q$ if
$$\|S_{t}f\|_{L^{q}(\mathbb{R}^N)}\leq
Ct^{\frac{N}{2m}(\frac{1}{q}-\frac{1}{p})}\|f\|_{L^{p}(\mathbb{R}^N)},$$ where
$C>0$, independent of $t$ and $f\in L^{2}(\mathbb{R}^{N})\cap
L^{p}(\mathbb{R}^{N})$. If $\{S_{t}\}_{t>0}$ satisfy an $L^p$-$L^p$ estimates, then $\{S_{t}\}_{t>0}$ are bounded on $L^p(\mathbb{R}^N)$
uniformly in $t$. In this case, we say $S_{t}$ is bounded on $L^p(\mathbb{R}^N)$.}
\end{defn}

Let us first state a useful result whose proof is easy and skipped here, one can see Hoffman and Martell \cite{H-M09} for a similar proof for special indexes.
\begin{lem}\label{le25}
If $\{T_{t}\}_{t>0}$ satisfy  $L^p$-$L^q$ estimates (resp.
off-diagonal estimates) and $\{S_{t}\}_{t>0}$ satisfy
$L^q$-$L^r$ estimates (resp. off-diagonal estimates), then $\{S_{t}T_{t}\}_{t>0}$ satisfy  $L^{p}$-$L^{r}$
estimates (resp. off-diagonal estimates).
\end{lem}

We also need another basic lemma related to the Sobolev embedding theorem.   Let $I_{m,N}$ to denote the following interval:
\begin{equation}\label{index}
I_{m,N}=\left\{
\begin{array}{ll}
{[}2,{2N\over N-2m}{]},\  & \hbox{$N>2m;$} \\
{[}2,\infty{)},\ \ &
\hbox{$N=2m;$} \\
{[}2,\infty{]},\ \ & \hbox{$N<2m.$}
\end{array}
\right.
\end{equation}
Denote $I'_{m,N}=\{p\in [1,2]:\ \ \frac{1}{p}+\frac{1}{p'}=1,\ p'\in I_{m,N}\}$.

\begin{lem}\label{embedding1}
Let $f\in W^{m,2}(\mathbb{R}^{N})$, $p\in I_{m,N}$ and $\theta={N\over2m}(1-{2\over p})\in[0,1]$. Then $f\in L^p(\mathbb{R}^{N})$ and there exists a constant $C_{m,N}>0$ such that the following inequality holds:
$$\|f\|_{L^{p}(\mathbb{R}^{N})}\le C_{m,N}\|(-\Delta)^{m/2}f\|^{\theta}_{L^{2}(\mathbb{R}^{N})}\|f\|^{1-\theta}_{L^{2}(\mathbb{R}^{N})}$$
\end{lem}

The following proposition deals with the relations between $L^{p}$ bounded, $L^{p}-L^{2}$ estimates and
$L^{p}-L^{2}$ off-diagonal estimates  for families of operators related to semigroup $e^{-tL}$, which is also useful in some other applications. Here and in the following, for any $1\leq p<\frac{N}{m}$, denote by $p^{*}=\frac{Np}{N-mp}$ the Sobolev exponent of $p$.

\begin{pro}\label{pro1'} Let $p\in [1,2)$ and $L=L_{0}+V\in \mathscr{A}_{m}$ where $L_{0}$ is a homogeneous elliptic operator of order 2m in divergence form. Denote $\{T_{t}\}_{t>0}$ by the following families of operators
$$\{e^{-tL}\}_{t>0},\ \ \{tLe^{-tL}\}_{t>0}\ \ \text{and}\ \
  \{\sqrt{t}\nabla^{m} e^{tL}\}_{t>0}.$$
Then

{\rm(i)}  If $T_{t}$ is bounded on $L^{p}(\mathbb{R}^N)$, then it satisfies the
$L^{p}$-$L^2$ estimates.

{\rm(ii)}  If $T_{t}$ satisfies the $L^{p}$-$L^2$ estimates, then
for all $p<q<2$ it satisfies the $L^{q}$-$L^2$ off-diagonal estimates.

{\rm(iii)}  If $T_{t}$ satisfies the  $L^{p}$-$L^2$ off-diagonal
estimates, then it is bounded on $L^p(\mathbb{R}^N)$.
\end{pro}

\begin{proof} Notice that the proofs for
conclusions (ii) and (iii) are essentially similar to the ones in
\cite{As} and \cite{D-D-Y1}. We hence only give the proof of (i).

We first prove that (i) is  true for $\{\sqrt{t}\nabla^m e^{-tL}\}_{t>0}$. Notice that
 $\sqrt{t}\nabla^m e^{-tL}$ is bounded on $L^{2}(\mathbb{R}^N)$ (see Theorem \ref{pro20}) and
$$\sqrt{t}\nabla^m e^{-tL}=\sqrt{2}\sqrt{\frac{t}{2}}\nabla^m e^{-\frac{t}{2}L}e^{-\frac{t}{2}L},$$
thus it suffices to prove that $e^{-tL}$ satisfies the
$L^{p}$-$L^2$ estimates.

If $N\leq 2m$, it follows from Lemma \ref{embedding1} that  for all $r\in [2,\infty)$ and $f\in L^{2}(\mathbb{R}^N)\cap L^{r}(\mathbb{R}^N)$,
\begin{eqnarray}
\|e^{-tL}f\|_{L^{r}(\mathbb{R}^N)}\leq \|\nabla^{m}e^{-tL}f\|^{\theta}_{L^{2}(\mathbb{R}^N)}\|e^{-tL}f\|^{1-\theta}_{L^{2}(\mathbb{R}^N)}\leq Ct^{-\frac{N}{2m}(1/2-1/r)}\|f\|_{L^{2}(\mathbb{R}^N)},\nonumber
\end{eqnarray}
where $\theta={N\over2m}(1-{2\over r})$. Thus, by a standard duality argument, we have that $e^{-tL}$ satisfies the $L^{p}-L^{2}$ estimate for all $p\in (1,2]$.

If $N>2m$, it follows that $p< 2<\frac{N}{m}$. Then by Lemma \ref{pro22} and duality, we obtain that $e^{-tL}$ satisfies  the $L^{p}-L^{2}$ estimate for all $p\in [\frac{2N}{N+2m},2]$.  Hence it suffices to consider the case $p<\frac{2N}{N+2m}$.  To do this,
notice that
$\sqrt{t}\nabla^m e^{-tL}$ is bounded on $L^{r}(\mathbb{R}^N)$ for all $r\in [p,2]$,
then it follows from the Sobolev embedding theorem that
\begin{eqnarray}\label{extra9}
\|e^{-tL}f\|_{L^{r^{\ast}}(\mathbb{R}^N)}\leq Ct^{-\frac{1}{2}}\|f\|_{L^{r}(\mathbb{R}^N)},\ \ \ \ \ r\in [p,2].
\end{eqnarray}
Let $r_{0}=p$ and $r_{j}=r^{*}_{j-1}=\frac{Nr_{0}}{N-mjr_{0}}$ for $j\in \mathbb{N}$ with $j< \frac{N}{r_{0}m}$, we have
\begin{eqnarray}\label{eq-ext-1}
\|e^{-j_{0}tL}f\|_{L^{r_{j}}(\mathbb{R}^N)}
&\leq& \|e^{-tL}\|_{L^{r_{j-1}}-L^{r_{j}}}\cdots\|e^{-tL}\|_{L^{r_{0}}-L^{r_{1}}}\|f\|_{L^{r_{0}}}\nonumber\\
&\leq& Ct^{-\frac{j}{2}}\|f\|_{L^{r_{0}}(\mathbb{R}^N)}.
\end{eqnarray}
We choose $j_{0}\geq 1$ such that
$\frac{2N}{N+2m}<r_{j_{0}}<2$, then it follows from (\ref{eq-ext-1}) that  $e^{-tL}$ satisfies the  $L^{p}-L^{r_{j_{0}}}$  estimates, which combined the fact
that $e^{-tL}$ satisfies the  $L^{r_{j_{0}}}-L^{2}$  estimate can finish the proof.

It remains to show that (i) holds for $\{e^{-tL}\}_{t>0}$ and $\{tLe^{-tL}\}_{t>0}$. In fact, the same argument as in the proof of \cite[Theorem 3.1]{D-D-Y1} can be applied to show that (i) holds for $\{e^{-tL}\}_{t>0}$. Moreover, by the identity $tLe^{-tL}=2\frac{t}{2}Le^{-\frac{t}{2}L}e^{-\frac{t}{2}L}$ and the fact that  $tLe^{-tL}$ is bounded on $L^{2}(\mathbb{R}^N)$, (i) is also true for $\{tLe^{-tL}\}_{t>0}$.  Hence we finish the whole proof.
\end{proof}

\begin{thm}\label{pro22}  Let $L=L_{0}+V\in \mathscr{A}_{m}$ where $L_{0}$ is a homogeneous elliptic operator of order 2m in divergence form. Then for any $p\in I_{m,N}$ (resp. $p\in I'_{m,N}$), the following statements hold:

{\rm (i)}\ $\{e^{-tL}\}_{t>0}$ and $\{tLe^{-tL}\}_{t>0}$ satisfy the  $L^{2}$-$L^{p}$ (resp. $L^{p}$-$L^{2}$) estimates.

{\rm (ii)} $\{e^{-tL}\}_{t>0}$ and $\{tLe^{-tL}\}_{t>0}$ satisfy $L^{2}$-$L^{p}$ (resp. $L^{p}$-$L^{2}$) off-diagonal estimates.

{\rm (iii)} $\{e^{-tL}\}_{t>0}$ and $\{tLe^{-tL}\}_{t>0}$ are bounded on $L^{p}(\mathbb{R}^{N})$.
\end{thm}

\begin{proof}
Notice that  (iii) can be easily obtained by Proposition \ref{pro1'} and a duality argument if (ii) holds.
On the other hand, since $\{tLe^{-tL}\}_{t>0}$ is bounded on $L^{2}(\mathbb{R}^{N})$ and  satisfies $L^{2}$ off-diagonal estimates, then by the identity $tLe^{-tH}=2(e^{-tL/2})(\frac{t}{2}Le^{-tL/2})$, Lemma \ref{le25} and duality, we only need to prove that the conclusions (i) and (ii) are true for $\{e^{-tL}\}_{t>0}$ and $p\in I_{m,N}$.


We begin with the proof of (ii). For $L\in\mathscr{A}_{m}$, there exist constants $0<\rho<\Lambda<\infty$ and $\mu\in[0,1)$ such that $L\in \mathscr{A}_{m}(\rho,\Lambda,\mu)$. Let $\lambda\in \mathbb{R}$,  $\phi\in \mathcal {E}_{m}(\mathbb{R}^{N})$ and $L_{\lambda\phi}$ be the perturbed operator defined in (\ref{perb-oper}). Write $g_{t}:=e^{-tL_{\lambda\phi}}g $ for each $g\in L^{2}(\mathbb{R}^{N})$, it follows that
$g_{t}\in D(H_{\lambda\phi})\subset W^{m,2}(\mathbb{R}^{N})$. Then by Lemma \ref{embedding1} we have that
\begin{eqnarray}\label{equation-2}
\|g_{t}\|_{L^{p}(\mathbb{R}^{N})}&\leq& C_{m,N}\|(-\Delta)^{m/2}g_{t}\|^{\theta}_{L^{2}(\mathbb{R}^{N})}\|g_t\|^{1-\theta}_{L^{2}(\mathbb{R}^{N})}
\end{eqnarray}
where $p\in I_{m,N}$ and $\theta={N\over2m}(1-{2\over p})\in [0,1]$. Notice that for all  $f\in D(\mathcal {Q})\subset W^{m,2}(\mathbb{R}^N)$, it is easy to show that
\begin{eqnarray}\label{eq25}
|\Re\mathcal {Q}_{\lambda\phi}(f,f)-\mathcal {Q}(f,f)|&\leq&|\mathcal {Q}_{\lambda\phi}(f,f)-\mathcal {Q}(f,f)|\nonumber\\
&\leq&\varepsilon\mathcal {Q}_{0}(f,f)+b_{\varepsilon}(1+\lambda^{2m})\|f\|^{2}_{L^{2}(\mathbb{R}^N)}\nonumber\\
&\leq&\frac{\varepsilon}{1-\mu}\mathcal {Q}(f,f)+b_{\varepsilon}(1+\lambda^{2m})\|f\|^{2}_{L^{2}(\mathbb{R}^N)},\nonumber
\end{eqnarray}
which combined Lemma \ref{le004} imply that
\begin{eqnarray}\label{L^2-L^p}
\|(-\Delta)^{m/2}g_{t}\|_{L^{2}(\mathbb{R}^{N})}\leq  C\big{(}\mathcal {Q}(f_{t}, f_{t})\big{)}^{1/2}
&\leq& C'\big{\{}\Re \mathcal{Q}_{\lambda\phi}(g_t,g_t)+(1+\lambda^{2m})\langle g_t, g_t\rangle\big{\}}^{1/2} \nonumber \\
&\leq& C'\big{\{} |\langle L_{\lambda\phi}g_t,g_t\rangle|+(1+\lambda^{2m})\langle g_t, g_t\rangle\big{\}}^{1/2} \nonumber \\
&\leq& C''\big{\{}t^{-1}+(1+\lambda^{2m})\big{\}}^{1/2}e^{c(1+\lambda^{2m})t}\|g\|_{L^{2}(\mathbb{R}^{N})}\nonumber\\
&\leq& D t^{-1/2}e^{(c+1)(1+\lambda^{2m})t}\|g\|_{L^{2}(\mathbb{R}^{N})},
\end{eqnarray}
 where the constants $C, C', C'',D,c$  are all independent of the choices of $L\in \mathscr{A}_{m}(\rho,\Lambda,\mu)$, $\lambda\in \mathbb{R}$ and  $\phi\in \mathcal {E}_{m}(\mathbb{R}^{N})$.  Therefore it follows from (\ref{equation-2}) and (\ref{L^2-L^p}) that
\begin{eqnarray}\label{L^2-L^pa}
 \|e^{-tL_{\lambda\phi}}g\|_{L^{p}(\mathbb{R}^{N})}&\leq& B t^{-{N\over2m}({1\over 2}-{1\over p})}e^{b(1+\lambda^{2m})t}\|g\|_{L^{2}(\mathbb{R}^{N})}.
\end{eqnarray}
 By (\ref{L^2-L^pa}) and the similar procedure as in the proof of Theorem \ref{thm1},  there exist constants $\omega, c, C>0$ independent of $L\in \mathscr{A}_{m}(\rho,\Lambda,\mu)$ such that
\begin{eqnarray}\label{eq38}
\|e^{-tL}f\|_{L^{p}(E)}\leq Ct^{-{N\over2m}({1\over 2}-{1\over p})}e^{\omega t}e^{-c\big{(}\frac {d(E,F)}{t^{1/2m}}\big{)}^{2m/(2m-1)}}\|f\|_{L^{2}(F)},
\end{eqnarray}
for arbitrary disjoint compact set $E,F\subset\mathbb{R}^N$, $t>0$ and $f\in L^{2}(\mathbb{R}^N)$ with supp$f\subset F$.  Finally, we can again use the the same scaling method as in the proof of Theorem \ref{thm1} to  get rid of the term $e^{\omega t}$ in the $(\ref{eq38})$ for any $p\in I_{m,N}$.
Thus we finish the proof of (ii).

Now we turn to prove (i). In fact, we can obtain (i) by using the same procedure as above. Hence we finish the whole proof.
\end{proof}


\begin{thm}\label{thm003} Let $L=L_{0}+V\in \mathscr{A}_{m}$ where $L_{0}$ is a homogeneous elliptic operator of order 2m in divergence form defined by (\ref{divergence}).
Then we have

{\rm (i)} $\{\sqrt{t}\nabla^m e^{-tL}\}_{t>0}$ satisfy the  $L^{p}-L^{2}$ estimates  for all $p\in I'_{m,N}$.

{\rm (ii)} $\{\sqrt{t}\nabla^m e^{-tL}\}_{t>0}$ satisfy  $L^{p}-L^{2}$ off-diagonal estimates  for all $p\in I'_{m,N}$.

{\rm (iii)} $\{\sqrt{t}\nabla^m e^{-tL}\}_{t>0}$ is bounded on $L^{p}(\mathbb{R}^{N})$ for all $p\in I'_{m,N}$.
\end{thm}

\begin{proof}
As done in the proof of Theorem \ref{pro22}, we just prove the case $N>2m$ for simplicity.
Notice that
$$\sqrt{t}\nabla^{m}e^{-tL}=\sqrt{2}(\sqrt{t/2}\nabla^{m}e^{-tL/2})(e^{-tL/2}),$$
then it follows from Proposition \ref{pro1'} and a duality process that all conclusions hold.
Hence the whole proof of theorem could be finished.
\end{proof}

\section{The Riesz transform for $q\leq 2$}\label{sec3'}

Let $L=L_{0}+V\in \mathscr{A}_{m}$ where $L_{0}$ is a homogeneous elliptic operator of order $2m$ in divergence form defined by (\ref{divergence}). Notice that by (\ref{condition}), we have
\begin{equation}\label{L2bound}
\|\nabla^m f\|_{L^2(\mathbb{R}^{N})}\le C\|(-\Delta)^{m/2}f\|_{L^2(\mathbb{R}^{N})}\le C'\|L^{1/2}f\|_{{L^2(\mathbb{R}^{N})}},
\end{equation}
which gives that ${\rm ker}(L^{1/2})=\{0\}$ and $\overline{{\rm Ran}(L^{1/2})}=L^2(\mathbb{R}^{N})$. Hence $\nabla^m L^{-1/2}$ can extend to a bounded operator on $L^2(\mathbb{R}^{N})$. However, it is not obvious if $\nabla^m L^{-1/2}$ defined by (\ref{Riesz}) is bounded on $L^{q}$ for $q\neq2$. In this section, we will apply the off-diagonal estimates to  investigate the $L^{q}$ boundedness of $\nabla^m L^{-1/2}$ for $q<2$.

Let us recall two important lemmas which treat  the boundedness of general  Calder\'{o}n-Zygmund type operators on $L^q$.
For a ball $B\subset \mathbb{R}^{N}$ and $\lambda>0$, we denote by
$\lambda B$ the ball with same center and radius $\lambda$ times
that of $B$. We set
$$S_{1}(B)=4B,\quad  S_j(B)=2^{j+1}B\setminus2^{j}B\  {\rm for}\ \ j\ge2.$$
Denote by $\mathcal {M}$ the Hardy-littlewood maximal operator
$$\mathcal {M}(f)(x)=\sup_{x\in B}\frac{1}{|B|}\int_{B}|f(y)|dy$$
where $B$ ranges over all open balls (or cubes) containing $x$.

\begin{lem}\label{le-1}  Let $p_{0}\in [1,2)$.
Suppose that $T$ is a sublinear operator of strong type $(2,2)$ and
$\{A_{r}\}_{r>0}$ is a family of linear operators acting on
$L^{2}(\mathbb{R}^N)$. If for $j\geq2$
\begin{equation} \label{cond25}
\begin{split}\Big{(}\frac{1}{|2^{j+1}B|}\int_{S_{j}(B)}|T(I-A_{r(B)})
f(x)|^{2}dx\Big{)}^{\frac{1}{2}}\leq
g(j)\Big{(}\frac{1}{|B|}\int_{B}|f(x)|^{p_{0}}dx\Big{)}^{\frac{1}{p_{0}}},
\end{split}
\end{equation}
and for $j\geq1$
\begin{equation} \label{cond26}
\begin{split}\Big{(}\frac{1}{|2^{j+1}B|}\int_{S_{j}(B)}|A_{r(B)}f(x)|^{2}dx\Big{)}^{\frac{1}{2}}\leq
g(j)\Big{(}\frac{1}{|B|}\int_{B}|f(x)|^{p_{0}}dx\Big{)}^{\frac{1}{p_{0}}},
\end{split}
\end{equation}
for all ball B with $r(B)$ the radius of B and all $f$ supported in
B. If $\sum_{j}g(j)2^{jN}<\infty$, then T is of weak type
$(p_{0},p_{0})$, with the bound depending only on the strong type
$(2,2)$ bound of T, $p_{0}$ and the sum $\sum_{j}g(j)2^{jN}$. Hence,
by interpolation T also is bounded on $L^{p}(\mathbb{R}^N)$ for
$p_{0}<p<2$.
\end{lem}

\begin{lem}\label{le006}Let $p_{0}\in [2,\infty)$. Suppose that $T$ is a sublinear operator
acting on $L^{2}(\mathbb{R}^n)$ and  $\{A_{r}\}_{r>0}$ is a family
of linear operators acting on $L^{2}(\mathbb{R}^n)$. Also assume
that
\begin{equation} \label{cond27}
\begin{split}\Big{(}\frac{1}{|B|}\int_{B}|T(I-A_{r(B)})f(x)|^{2}dx\Big{)}^{\frac{1}{2}}\leq
C(\mathcal {M}(|f|^{2}))^{\frac{1}{2}}(y)
\end{split}
\end{equation}
and
\begin{equation} \label{cond28}
\begin{split}\Big{(}\frac{1}{|B|}\int_{B}|TA_{r(B)}f(x)|^{p_{0}}dx\Big{)}^{\frac{1}{p_{0}}}\leq
C(\mathcal {M}(|Tf|^{2}))^{\frac{1}{2}}(y)
\end{split}
\end{equation}
for all $f\in L^{2}$, all ball B and all $y\in B$ where $r(B)$ is
the radius of B. If $2<p<p_{0}$ and $Tf\in L^{p}$ as $f\in L^{p}$,
then T is strong type $(p,p)$. That is, for all $f\in
L^{2}(\mathbb{R}^n)\cap L^{p}(\mathbb{R}^n)$,
$$\|Tf\|_{L^{p}(\mathbb{R}^n)}\leq c\|f\|_{L^{p}(\mathbb{R}^n)},$$ where $c$ depends
only on n, p, and $p_{0}$ and C.
\end{lem}

Lemmas \ref{le-1} and \ref{le006} are essentially due to \cite[Theorem 1.1, p.~920]{B-K03}
and \cite[Theorem 2.1, p.~923]{A-C-D-H}, respectively. See \cite{A2007}
for the proofs and nice comments on them.

\vskip 0.3cm

\subsection{The $L^{q}$ boundedness of $\nabla^{m}L^{-1/2}$ for $q\leq2$}

We consider the $L^{q}$ boundedness of $\nabla^{m}L^{-1/2}$ for $q<2$ in this subsection. First of all, let us prove the following proposition which  connects  the Riesz transform $\nabla^{m}L^{-1/2}$ and $e^{-tL}$.
\begin{pro}\label{prop-2}
Let $L=L_{0}+V\in \mathscr{A}_{m}$ where $L_{0}$ is a homogeneous elliptic operator of order 2m in divergence form defined by (\ref{divergence}). Then the following statements hold.

{\rm (i)}\ If $e^{-tL}$ satisfies $L^{p}-L^{2}$ off-diagonal estimates for some $p\in [1,2]$, then the Riesz transform $\nabla^{m}L^{-1/2}$ is of weak type $(p,p)$ and  bounded on $L^{q}$ for all $q\in (p,2]$.

{\rm (ii)}\ If $\nabla^{m}L^{-1/2}$ is  bounded on
$L^{p}$ for some $p\in (1,2)$, then $e^{-tL}$ satisfies  $L^{p}$-$L^{2}$ estimates.
\end{pro}

\begin{proof}
We prove (i) first. Notice that $\nabla^{m}L^{-1/2}$  is bounded on $L^{2}(\mathbb{R}^{N})$, then by Lemma \ref{le-1}, it suffices
 to verify that (\ref{cond25}) and (\ref{cond26}) hold for $T=\nabla^{m}L^{-1/2}$ and $p_{0}=p$. To this end,
let $B$ be a ball and $r=r(B)$ its radius. Choose $A_{r}={\bf I}-({\bf I}-e^{-r^{2m}L})^{M}$
for $M\in \mathbb{N}$ and $M>\frac N{4m}$.

We first verify (\ref{cond26}).
Let $f$ be supported in
$B$. Notice that $A_{r}=\sum^{M}_{\ell=1}C_{M,\ell}e^{-\ell
r^{2m}L}$ and $e^{-tL}$ satisfies the $L^{p}$-$L^2$ off-diagonal estimates, then for any $j\ge 1$, we obtain that
\begin{eqnarray}\label{cond232'}
\Big{(}\frac{1}{|2^{j+1}B|}\int_{S_{j}(B)}|A_rf(x)|^{2}dx\Big{)}^{\frac{1}{2}}
&\le&\sum^{M}_{\ell=1}C_{M,\ell}\Big{(}\frac{1}{|2^{j+1}B|}\int_{S_{j}(B)}|e^{-\ell r^{2m}L}f(x)|^{2}dx\Big{)}^{\frac{1}{2}}\nonumber\\
&\leq&\sum^{M}_{\ell=1}C_{M,\ell}2^{-\frac{jN}{2}}|B|^{-\frac{1}{2}}\|e^{-\ell r^{2m}L}f\|_{L^{2}(S_{j}(B))}\nonumber\\
&\leq&C_M2^{-\frac{jN}{2}}|B|^{-\frac{1}{2}}r^{-N(\frac{1}{p}-\frac{1}{2})}
e^{-c\big(\frac{d(B,S_{j}(B))^{2m}}{r^{2m}}\big)^{\frac{1}{2m-1}}}\|f\|_{L^{p}(B)}\nonumber\\
&\leq&g_1(j)\Big{(}\frac{1}{|B|}\int_{B}|f(x)|^{p}dx\Big{)}^{\frac{1}{p}},\nonumber
\end{eqnarray}
where  $g_1(j)=C_M2^{-\frac{jN}{2}}e^{-c_M2^{\frac{2jm}{2m-1}}}$ as $j\ge 2$ and $g_1(1)=C_M2^{-{N\over 2}}$.
Thus for $j\geq1$,
\begin{eqnarray}\label{eq54}
\Big{(}\frac{1}{|2^{j+1}B|}\int_{S_{j}(B)}|A_{r(B)}f(x)|^{2}dx\Big{)}^{\frac{1}{2}}\leq
g_1(j)\Big{(}\frac{1}{|B|}\int_{B}|f(x)|^{p}dx\Big{)}^{\frac{1}{p}}
\end{eqnarray}
holds with $\sum_{j=1}^\infty g_1(j)2^{jN}<\infty$.

It remains to establish  (\ref{cond25}). Notice that by the same procedure as in the proof of Auscher \cite[Lemma 4.4]{A2007},
\begin{equation}\label{equation-1}
\begin{split}\|\nabla^{m}L^{-1/2}({\bf I}-e^{-r^{2m}L})^{M}f\|_{L^{2}(S_{j}(B))}
\leq Cr^{(\frac{N}{2}-\frac{N}{p})}2^{-2mMj}\|f\|_{L^{p}(B)}
\end{split}
\end{equation}
holds for given $p\in[1,2]$ as in assumptions. Thus, for $j\geq2$ we get that
\begin{eqnarray}
&&\Big{(}\frac{1}{|2^{j+1}B|}\int_{S_{j}(B)}|\nabla^{m}L^{-1/2}(\textbf{I}-A_{r(B)})f(x)|^{2}dx\Big{)}^{\frac{1}{2}}\nonumber\\
&\leq&C2^{-\frac{jN}{2}}|B|^{-\frac{1}{2}}\|\nabla^{m}L^{-1/2}(\textbf{I}-e^{-r^{2m}L})^{M}f\|_{L^{2}(S_{j}(B))}\nonumber\\
&\leq&C2^{-\frac{jN}{2}}|B|^{-\frac{1}{2}}r^{-N(\frac{1}{p}-\frac{1}{2})}2^{-2mMj}\|f\|_{L^{p}(B)}\nonumber\\
&\leq&g_2(j)\Big{(}\frac{1}{|B|}\int_{B}|f(x)|^{p}dx\Big{)}^{\frac{1}{p}},\nonumber
\end{eqnarray}
where $g_2(j)=C2^{-j(\frac{N}{2}+2mM)}$. Then $\sum_{j=2}^\infty g_2(j)2^{jN}<\infty$ when $M\in
\mathbb{N}$ and $M>\frac{N}{4m}$, which means that
(\ref{cond25}) holds for $T=\nabla^{m}L^{-1/2}$ and $p_{0}=p$. We hence finish the proof of (i).

We now turn to prove the statement (ii). Let $L\in\mathscr{A}_{m}$, it follows from Theorem \ref{pro22} and a duality argument that $e^{-tL}$ satisfies
$L^{p}-L^{2}$ estimates for all $p\in I'_{m,N}$ (see (\ref{index}) its definition).
Thus, it suffices to prove (ii) for $p<\frac{2N}{N+2m}$ when $N>2m$.

Notice that $\nabla^{m}L^{-1/2}$ is bounded on $L^{r}(\mathbb{R}^{N})$ for all $r\in [p,2]$, then
it follows from the Sobolev embedding theorem that
\begin{eqnarray}
L^{-1/2}:\ L^{r}(\mathbb{R}^{N})\rightarrow\ L^{r^{\ast}}(\mathbb{R}^{N}),\ \ \ \ \ r\in [p,2].\nonumber
\end{eqnarray}
Now let $r_{0}=p$ and $r_{j}=r^{\ast}_{j-1}=\frac{Nr_{0}}{N-mjr_{0}}$ for $j\in \mathbb{N}$  with $j< \frac{N}{r_{0}m}$, it is easy to see that
\begin{eqnarray}\label{eq55}
L^{-j/2}:\ L^{p}(\mathbb{R}^{N})\rightarrow\ L^{r_{j}}(\mathbb{R}^{N}).
\end{eqnarray}
By choosing  $j_{0}\geq 1$ such that $\frac{2N}{N+2m}<r_{j_{0}}<2$, it follows from (\ref{eq55}) that $L^{-j_{0}/2}$ satisfies $L^{p}-L^{r_{j_{0}}}$ estimate.
Write
\begin{eqnarray}
e^{-tL}=(e^{-\frac{t}{2}L}L^{j_{0}/2})e^{-\frac{t}{2}L}L^{-j_{0}/2},\nonumber
\end{eqnarray}
we have
 $$\|e^{-\frac{t}{2}L}L^{-j_{0}/2}\|_{L^{p}-L^{2}}\leq \|e^{-\frac{t}{2}L}\|_{L^{r_{j_{0}}}-L^{2}}\|L^{-j_{0}/2}\|_{L^{p}-L^{r_{j_{0}}}}\leq Ct^{-\frac{N}{2m}(1/r_{j_{0}}-1/2)}.$$
On the other hand, by the bounded holomorphic calculus of $L$
on $L^{2}(\mathbb{R}^{N})$,  $e^{-\frac{t}{2}L}L^{j_{0}/2}$ is $L^{2}-L^{2}$ bounded by $Ct^{-j_{0}/2}$. Therefore
\begin{eqnarray}
\|e^{-tL}\|_{L^{p}-L^{2}}&\leq& \|e^{-\frac{t}{2}L}L^{j_{0}/2}\|_{L^{2}-L^{2}}\|e^{-\frac{t}{2}L}L^{-j_{0}/2}\|_{L^{p}-L^{2}}\leq Ct^{-\frac{N}{2m}(1/p-1/2)}.\nonumber
\end{eqnarray}
We finish the proof.
\end{proof}

For $L\in \mathscr{A}_{m}$, by using Proposition \ref{prop-2} and the off-diagonal estimates for $e^{-tL}$, we now show the $L^{q}$ boundedness of $\nabla^{m}L^{-1/2}$ for $q\leq2$.
\vskip0.3cm
{\bf \emph{The proof of Theorem \ref{thm4}}:}
Let $L$ be defined as in Theorem \ref{thm4}. Then it follows from Theorem \ref{pro22}  that $e^{-tL}$ satisfies $L^{q}-L^{2}$ off-diagonal estimates for all $q\in I'_{m,N}$. Thus, we can finish the proof of Theorem \ref{thm4} by using (i) of Proposition \ref{prop-2}.
\vskip0.3cm


\subsection{Sharpness of Theorem \ref{thm4}}

In this subsection, we will discuss the sharpness of the interval $(\frac{2N}{N+2m}\vee1,2]$ in Theorem \ref{thm4}. It has been pointed out in Remark \ref{rem1-2} that when $N\leq 2m$, the lower bound of $(1,2]$ in Theorem \ref{thm4} is optimal with respect to $\mathscr{A}_{m}$ and the sharpness of the upper bound  of $(1,2]$ is still unknown to us.
However, for $N>2m$, we have the following sharpness result.

\begin{thm}\label{sharp-1}
Let $m\geq2$ and $N>2m$. The interval $(\frac{2N}{N+2m},2]$ in Theorem \ref{thm4} is sharp in the following sense: for given $q<\frac{2N}{N+2m}$ or $q>2$, there exists an  operator $L\in \mathscr{A}_{m}$ such that $\nabla^{m}L^{-1/2}$ is not bounded on $L^{q}(\mathbb{R}^{N})$.
\end{thm}

Before proving Theorem \ref{sharp-1}, we first see the following lemma which essentially belongs to Davies \cite{D97-1}.

\begin{lem}\label{lemma-21}
Let $m\geq2$ and $N>2m$. For given $q<\frac{2N}{N+2m}$ or $q>\frac{2N}{N-2m}$, there exist a self-adjoint elliptic operator $L_{0}$ of order $2m$ in divergence form with bounded coefficients satisfying (\ref{strong}) such that $e^{-tL_{0}}$ cannot be extended from $L^{2}(\mathbb{R}^{N})\cap L^{q}(\mathbb{R}^{N})$ to a uniformly bounded operator on $L^q(\mathbb{R}^{N})$ for any $t>0$.
\end{lem}

\begin{proof}
In fact, let $L_{0}$ be the operator $A$ defined in Davies \cite[Theorem 5]{D97-1}, it follows that $L_{0}$ is an uniformly non-negative self-adjoint uniformly operator of order $2m$ with coefficients $a_{\alpha,\beta}$. Moreover, $a_{\alpha,\beta}$ are bounded smooth functions of $x$ on $\mathbb{R}^{N}\backslash\{0\}$ and $a_{\alpha,\beta}(sx)=a_{\alpha,\beta}(x)$ for all $s>0$ and $x\in \mathbb{R}^{N}$. Then by Davies \cite[Theorem 10]{D97-1}, we have that for all $p\overline{\in} [\frac{2N}{N+2m},\frac{2N}{N-2m}]$, the $e^{-tL_{0}}$ cannot be extended from $L^{2}(\mathbb{R}^{N})\cap L^{q}(\mathbb{R}^{N})$ to a uniformly bounded operator on $L^q(\mathbb{R}^{N})$ for any $t>0$.
\end{proof}

{\bf \emph{The Proof of Theorem \ref{sharp-1}:}}
We first prove that the lower bound $\frac{2N}{N+2m}$ is sharp. Actually, by using (ii) of Proposition \ref{prop-2} and Proposition \ref{pro1'}, we only need to prove  that for $N>2m$ and given $q<\frac{2N}{N+2m}$,
there exists an operator $L\in\mathscr{A}_{m}$ such that the semigroup $e^{-tL}$ cannot be extended from $L^{2}(\mathbb{R}^{N})\cap L^{q}(\mathbb{R}^{N})$ to an uniformly bounded operator on $L^q(\mathbb{R}^{N})$.
In fact, let $L=L_{0}+V$ where $L_{0}$ be the operator defined in Lemma \ref{lemma-21} and $V=0$.
It follows from Lemma \ref{lemma-21} that for all $q<\frac{2N}{N+2m}$, the $e^{-tL}$ cannot be extended from $L^{2}(\mathbb{R}^{N})\cap L^{q}(\mathbb{R}^{N})$ to a uniformly bounded operator on $L^q(\mathbb{R}^{N})$ for any $t>0$. Hence, we finishes the proof for $q<\frac{2N}{N+2m}$.

Now we turn to prove that the upper bound $2$ is sharp. The proof would involve some results of next section.
Without loss of generality, we assume that $2<q<\frac{N}{m}$ in the whole proof. Choosing a fixed $\frac{2N}{N-2m}<p<\frac{Nq}{N-qm}$, it follows from Lemma  \ref{lemma-21} that  there exists an operator $L\in\mathscr{A}_{m}$ such that $e^{-tL}$ cannot be extended from $L^{2}(\mathbb{R}^{N})\cap L^{p}(\mathbb{R}^{N})$ to a uniformly bounded operator on $L^p(\mathbb{R}^{N})$,
hence the Riesz transform $\nabla^{m}L^{-1/2}$ is not bounded on $L^{q}(\mathbb{R}^{N})$.
Otherwise, it follows from (i) of Proposition \ref{pro31} and Proposition \ref{le52} that $\sqrt{t}\nabla^{m}e^{-tL}$ is uniformly bounded on $L^{r}$ for all $2<r<q$. Then by Lemma \ref{le005}, we can obtain that $e^{-tL}$ is uniformly bounded on $L^{r}$ for all $2\leq r<\frac{Nq}{N-qm}$, which makes contradiction since $e^{-tL}$ is not uniformly bounded on $L^p(\mathbb{R}^{N})$ for given $\frac{2N}{N-2m}<p<\frac{Nq}{N-qm}$. Hence, we finish the proof.

\section{The Riesz transform for $q>2$}

Let $L=L_{0}+V\in \mathscr{A}_{m}$ where $L_{0}$ is a homogeneous elliptic operator of order $2m$ in divergence form defined by (\ref{divergence}), the $L^q$ boundedness of $\nabla^{m}L^{-1/2}$ for $q>2$ will be investigate  in this section.

We begin  with the $L^p$-regularity of $\sqrt{t}\nabla^{m} e^{-tL}$ for $p>2$, which will help us build a bridge between the boundedness of $\nabla^{m}L^{-1/2}$ on $L^{q}$ for $q>2$ and the off-diagonal estimates of  $\sqrt{t}\nabla^{m} e^{-tL}$, see subsection \ref{sec4-1} below. On the other hand, notice that extra conditions on $L\in \mathscr{A}_{m}$ are necessary due to Shen's counterexample in \cite{Sh95} and the sharpness result Theorem \ref{sharp-1}, thus two types of conditions on both $L_{0}$ and $V$ will be introduced and discussed in subsections \ref{sec4-2} and \ref{sec4-3}, respectively.

\subsection{The $L^p$-regularity of $\sqrt{t}\nabla^{m} e^{-tL}$ for $p>2$}\label{sec4-1}

We recall in Proposition \ref{prop-2} that
the $L^q$ ($q<2$) boundedness of Riesz transform $\nabla^{m}L^{-1/2}$ is essentially equivalent to the $L^{p}-L^{2}$ estimates (off-diagonal estimates) of $e^{-tL}$. In this subsection, we will establish the equivalence between
the $L^p$ $(p>2)$ properties of $\sqrt{t}\nabla^{m}e^{-tL}$ and the $L^q$  boundedness of $\nabla^{m}L^{-1/2}$ for $q>2$. To this end, we first consider the family of operators $\{\nabla^{m}e^{-tL}\}_{t>0}$ on $L^{p}(\mathbb{R}^N)$ for $p>2$.

\begin{lem}\label{le005}
Let $N> 2m\geq4$ and  $L=L_{0}+V\in \mathscr{A}_{m}$ where $L_{0}$ is a homogeneous elliptic operator of order 2m in divergence form defined by (\ref{divergence}).
Assume that $\{\sqrt{t}\nabla^{m} e^{-tL}\}_{t>0}$ is uniformly  bounded  operators on $L^{q}(\mathbb{R}^{N})$ for some $2\leq q<\frac{N}{m}$, then the operators $\{e^{-tL}\}_{t>0}$ is bounded on $L^{p}(\mathbb{R}^{N})$ for all $2\leq p<q^{\ast}=\frac{Nq}{N-mq}$.
\end{lem}

\begin{proof}It follows from Theorem \ref{pro22}  that $e^{-tL}$ satisfies $L^{2}-L^{q}$ for all
$q\in [2,\frac{2N}{N-2m}]$ if $N>2m$. Thus by Proposition \ref{pro1'} and duality, we only need to prove this lemma for $\frac{2N}{N-2m}<q<\frac{N}{m}$.

For fixed $q\in[2,\frac{N}{m})$, interpolating by Riesz-Thorin theorem the $L^q$ and $L^2$ boundedness of $\sqrt{t}\nabla e^{-tL}$, we have that $\sqrt{t}\nabla e^{-tL}$
is bounded on $L^{r}$ for all
$2\leq r\leq q$.
 Then it follows from the Sobolev embedding theorem that
\begin{eqnarray}\label{eq26'}
\|e^{-tL}\|_{L^{r}-L^{r^{*}}}\leq Ct^{-\frac{1}{2}}, \ \ \ \  r\in [2,q].
\end{eqnarray}
Notice that there exists constants $j_{0}\in \mathbb{N}$ and $2\leq r_{0}<\frac{2N}{N-2m}$ such that
$q^{\ast}=\frac{Nr_{0}}{N-mj_{0}r_{0}}<\infty$ and $e^{-tL}$ satisfies $L^{2}$-$L^{r_{0}}$ estimates (see Theorem \ref{pro22}), then let $r_{j}=(r_{j-1})^{\ast}$ $(j=1,2,\ldots, j_{0})$, it follows from (\ref{eq26'}) that
\begin{eqnarray}
\|e^{-(j_{0}+1)tL}\|_{L^{2}-L^{q^{\ast}}}\leq \|e^{-tL}\|_{L^{2}-L^{r_{0}}}\|e^{-tL}\|_{L^{r_{0}}-L^{r_{1}}}\cdots\|e^{-tL}\|_{L^{r_{j_{0}-1}}-L^{r_{j_{0}}}}\leq Ct^{-\frac{N}{2m}(\frac{1}{2}-\frac{1}{q^\ast})},\nonumber
\end{eqnarray}
 which combined Proposition \ref{pro1'} and duality can finish the proof.
\end{proof}

The following lemma studies the relations between $L^{p}$ bounded, $L^{2}-L^{p}$ estimates and $L^{2}-L^{p}$ off-diagonal estimates for operator $\sqrt{t}\nabla^{m}e^{-tL}$ with $p>2$.

\begin{pro}\label{le52} Let $p\in (2,\infty]$, $N>2m$ and $L=L_{0}+V\in \mathscr{A}_{m}$ where $L_{0}$ is a homogeneous elliptic operator of order 2m in divergence form defined by (\ref{divergence}).  Then the following statements hold.

{\rm (i)} If $\{\sqrt{t}\nabla^{m}e^{-tL}\}_{t>0}$ is bounded on  $L^{p}(\mathbb{R}^{N})$, then it satisfies the $L^{2}-L^{p}$ estimates.

{\rm (ii)} If $\{\sqrt{t}\nabla^{m}e^{-tL}\}_{t>0}$ satisfies the $L^{2}-L^{p}$ estimates, then it satisfies the $L^{2}-L^{q}$ off-diagonal estimates for $2<q<p$.

{\rm (iii)} If $\{\sqrt{t}\nabla^{m}e^{-tL}\}_{t>0}$ satisfies the $L^{2}-L^{p}$ off-diagonal estimates, then it is  bounded  on  $L^{p}(\mathbb{R}^{N})$.
\end{pro}

\begin{proof}
We prove (i) first. When $N>2m$, by the identity
$\sqrt{t}\nabla^{m}e^{-tL}=\sqrt{2}\sqrt{\frac{t}{2}}\nabla^{m}e^{-\frac{t}{2}L}e^{-\frac{t}{2}L}$, the fact that $\sqrt{t}\nabla^{m}e^{-tL}$ is bounded on  $L^{p}(\mathbb{R}^{N})$ and Theorem \ref{pro22}, it suffices to prove that $e^{-tL}$  satisfies  $L^{2}-L^{p}$ estimates for $p>\frac{2N}{N-2m}$.
We will prove it by splitting $p$ into two parts.

Assume that $\frac{2N}{N-2m}<p<\infty$, then there exists a constant $2<q<\frac{N}{m}$ such that $q<p<q^{\ast}=\frac{Nq}{N-mq}$. On the other hand, since $2<q< p$, it follows that  $\sqrt{t}\nabla^{m}e^{-tL}$ is bounded on  $L^{q}(\mathbb{R}^{N})$. Thus, by the proof of  Lemma \ref{le005}, we can obtain that $e^{-tL}$  satisfies  the $L^{2}-L^{q^{\ast}}$ estimates, which implies the result as desired.

Assume that $p=\infty$, we have to prove that $e^{-tL}$ satisfies the $L^{2}-L^{\infty}$ estimates. Notice that
$\sqrt{t}\nabla^{m}e^{-tL}$ is bounded on $L^{q}(\mathbb{R}^N)$ for all $2\leq q\leq \infty$, then
by choosing $\frac{N}{m}<q<\frac{N}{m-1}$ and the Sobolev embedding theorem, we have $W^{m,q}(\mathbb{R}^{N})\hookrightarrow C^{0,\gamma}(\mathbb{R}^{N})$ with $\gamma=m-\frac{N}{q}$. Hence, for $t>0$ and $x,y\in \mathbb{R}^{N}$
\begin{eqnarray}
|e^{-tL}f(x)-e^{-tL}f(y)|\leq C\|\nabla^{m}e^{-tL}\|_{L^{q}(\mathbb{R}^{N})}|x-y|^{\gamma}\leq Ct^{-\frac{1}{2}-\frac{N}{2m}(\frac{1}{2}-\frac{1}{q})}\|f\|_{L^{2}(\mathbb{R}^{N})}|x-y|^{\gamma}.\nonumber
\end{eqnarray}
where we use the fact that $\sqrt{t}\nabla^{m}e^{-tL}$ satisfies $L^{2}-L^{q}$ estimates for all $\frac{N}{m}<q<\frac{N}{m-1}$ as we just proved it. Now fix $x\in \mathbb{R}^{N}$, let $B$ be the ball with center $x$ and radius $r=t^{\frac{1}{2m}}$, we average the square of the above inequality to get
\begin{eqnarray}
|e^{-tL}f(x)|\leq C|B|^{-\frac{1}{2}}\|e^{-tL}f\|_{L^{2}(B)}+Cr^{\gamma}t^{-\frac{1}{2}-\frac{N}{2m}(\frac{1}{2}-\frac{1}{q})}\|f\|_{L^{2}(\mathbb{R}^{N})}
\leq Ct^{-\frac{N}{4m}}\|f\|_{L^{2}(\mathbb{R}^{N})},\nonumber
\end{eqnarray}
which implies that $e^{-tL}$ is bounded from  $L^{2}$ to $L^{\infty}$.
 Therefore, we finish the proof of (i).

Since we have proved $\sqrt{t}\nabla^{m}e^{-tL}$ $(t>0)$ satisfies $L^{2}$ off-diagonal estimate. Thus the statement (ii) follows easily by interpolation.

The statement (iii) can be obtained by applying \cite[Lemma 3.3]{A2007} to $T=e^{-tH^{\ast}}(\nabla^{m})^{\ast}$ and the duality argument.
\end{proof}

Now we turn to show the main result of this subsection which connects $\nabla^{m}L^{-\frac{1}{2}}$ with  $\sqrt{t}\nabla^{m}e^{-tL}$ on $L^{q}$ for $q>2$.

\begin{pro}\label{pro31}
Let $L=L_{0}+V\in \mathscr{A}_{m}$ where $L_{0}$ is a homogeneous elliptic operator of order 2m in divergence form defined by (\ref{divergence}). Then:

{\rm (i)} If $\nabla^{m}L^{-\frac{1}{2}}$ is bounded on $L^{p}$ for $2<p<\infty$, then $\{\sqrt{t}\nabla^{m}e^{-tL}\}_{t>0}$
satisfies $L^{2}-L^{q}$ off-diagonal estimates for all $2<q<p$.

{\rm (ii)} If $N>2m$ and $\{\sqrt{t}\nabla^{m}e^{-tL}\}_{t>0}$ satisfies $L^{2}-L^{p}$  off-diagonal  estimates for $2<p<\infty$, then $\nabla^{m}L^{-\frac{1}{2}}$ is bounded on $L^{q}$ with $2<q<p$.
\end{pro}

\begin{proof}
We first prove (i). By assumptions, it is easy to see that $\nabla^{m}L^{-\frac{1}{2}}$ is bounded on $L^{r}$ for $2\leq r\leq p$, which combined the Sobolev embedding theorem imply that
\begin{eqnarray}\label{equation-3}
\|L^{-\frac{1}{2}}f\|_{L^{r^{\ast}}}\leq C\|\nabla^{m}L^{-\frac{1}{2}}f\|_{L^{r}}\leq C\|f\|_{L^{r}},\ \ 2\leq r\leq p,\ r<\frac{N}{m},  \ r^{\ast}=\frac{Nr}{N-rm}.
\end{eqnarray}
We choose constants $j_{0}\in \mathbb{N}$ and $2\leq r_{0}\leq \frac{2N}{N-2m}$ such that $p=\frac{Nr_{0}}{N-mj_{0}r_{0}}<\infty$ and let  $r_{j}=(r_{j-1})^{\ast}$ for  $j=1,2,\cdots, j_0$. It follows from (\ref{equation-3}) that
\begin{eqnarray}\label{eq56}
L^{-j_{0}/2}:\ \ \ L^{r_{0}}(\mathbb{R}^{N})\rightarrow\ L^{r_{j_{0}}}(\mathbb{R}^{N}).
\end{eqnarray}
Now write
\begin{eqnarray}\label{eq4-2}
e^{-tL}f=L^{-j_{0}/2}e^{-\frac{t}{2}L}(L^{j_{0}/2}e^{-\frac{t}{2}L})f.
\end{eqnarray}
Notice that $L^{j_{0}/2}e^{-\frac{t}{2}L}$ is bounded on $L^{2}$ with
bound $Ct^{-j_{0}/2}$ and  $e^{-\frac{t}{2}L}$
satisfies the $L^{2}$-$L^{r_{0}}$ estimates (see Theorem \ref{pro22}), then it follows from (\ref{eq56}) and (\ref{eq4-2}) that $e^{-tL}$ satisfies $L^{2}-L^{p}$ estimates.
 Thus, we have
\begin{eqnarray}
\|\nabla^{m}e^{-tL}\|_{L^2-L^p}&=&\|\nabla^{m}L^{-\frac{1}{2}}e^{-\frac{t}{2}L}(L^{\frac{1}{2}}e^{-\frac{t}{2}L})\|_{L^2-L^p}\nonumber\\
&\leq&\|\nabla^{m}L^{-\frac{1}{2}}\|_{L^{p}-L^{p}}
\|e^{-\frac{t}{2}L}\|_{L^{2}-L^{p}}\|L^{\frac{1}{2}}e^{-\frac{t}{2}L}\|_{L^2-L^2}\leq Ct^{-1/2},\nonumber
\end{eqnarray}
which combined Proposition \ref{le52} finish the proof.

Let us turn to the proof of (ii). For given $2<p<\infty$ and  every $2<q<p$, it suffice to apply Lemma \ref{le006} to the operator $T=\nabla^{m}L^{-\frac{1}{2}}$ with $q_{0}$, where $2<q<q_{0}<p$. For every ball open $B$ with radius $r=r(B)$
and $A_{r}={\bf I}-({\bf I}-e^{-r^{2m}L})^{M}$
where $M\in \mathbb{N}$ would be chosen later, we will finish the proof of (ii) in the following two steps.

Step 1. We show that for $f\in L^{q_{0}}(\mathbb{R}^{N})\cap L^{2}(\mathbb{R}^{N})$
\begin{eqnarray}\label{eq57}
\|\nabla^{m}L^{-\frac{1}{2}}({\bf I}-e^{-r^{2m}L})^{M}f\|_{L^{2}(B)}\leq C|B|^{1/2}\sum_{j\geq1}g(j)
\Big(\frac{1}{|2^{j+1}B|}\int_{2^{j+1}B}|f|^{2}dx\Big)^{\frac{1}{2}},
\end{eqnarray}
with $g(j)=2^{j(\frac{N}{2}-2mM)}$. Let $S_{j}(B)$ $(j\geq1)$ be defined as in Section \ref{sec3'}. Then by Minkowski inequality we have
\begin{eqnarray}
\|\nabla^{m}L^{-\frac{1}{2}}({\bf I}-e^{-r^{2m}L})^{M}f\|_{L^{2}(B)}\leq \sum_{j\geq1}
\|\nabla^{m}L^{-\frac{1}{2}}({\bf I}-e^{-r^{2m}L})^{M}(\chi_{S_{j}(B)}f)\|_{L^{2}(B)}.\nonumber
\end{eqnarray}
For $j=1$, by the $L^{2}$ boundedness of $\nabla^{m}L^{-\frac{1}{2}}$ and $e^{-tL}$, we have
\begin{eqnarray}
\|\nabla^{m}L^{-\frac{1}{2}}({\bf I}-e^{-r^{2m}L})^{M}(\chi_{S_{1}(B)}f)\|_{L^{2}(B)}\leq C\|f\|_{L^{2}(B)}
\leq C|4B|^{1/2}\Big(\frac{1}{|4B|}\int_{4B}|f|^{2}dx\Big)^{\frac{1}{2}}.\nonumber
\end{eqnarray}
When $j\geq2$, it follows from the same idea as in the proof of (\ref{equation-1}) that
\begin{eqnarray}
\|\nabla^{m}L^{-\frac{1}{2}}({\bf I}-e^{-r^{2m}L})^{M}(\chi_{S_{j}(B)}f)\|_{L^{2}(B)}\leq C2^{-2mMj}\|f\|_{L^{2}(S_{j}(B))},\nonumber
\end{eqnarray}
which implies (\ref{eq57}) immediately. By choosing $M>\frac{N}{4m}$ in (\ref{eq57}), we can establish (\ref{cond27}).

Step 2. We show that (\ref{cond28}) holds. In fact, Notice that $A_{r}=\sum^{M}_{\ell=1}C_{M,\ell}e^{-\ell
r^{2m}L}$, we first prove that
\begin{eqnarray}\label{eq58}
\Big(\frac{1}{|B|}\int_{B}\Big|\nabla^{m}e^{-\ell
r^{2m}L}f(x)\Big|^{q_{0}}dx\Big)^{\frac{1}{q_{0}}}\leq C\sum_{j\geq1}g(j)\Big(\frac{1}{|2^{j+1}B|}\int_{2^{j+1}B}|\nabla^{m}f(x)|^{2}dx\Big)^{\frac{1}{2}},
\end{eqnarray}
for all $\ell=1,\ldots,M$ with $\sum_{j\geq1}g(j)<\infty$. Let $S_{j}(B)$ $(j\geq1)$ be defined as above. For $j=1$, by the assumption and Proposition \ref{le52},
we have $\sqrt{t}\nabla^{m}e^{-tL}$ also satisfies the $L^{2}-L^{q_{0}}$ estimate. Thus
\begin{eqnarray}
\Big(\frac{1}{|B|}\int_{B}\Big|\nabla^{m}e^{-\ell
r^{2m}L}(\chi_{S_{1}(B)}f)(x)\Big|^{q_{0}}dx\Big)^{\frac{1}{q_{0}}}\leq C|B|^{-\frac{1}{q_{0}}}r^{m+N(\frac{1}{q_{0}}-\frac{1}{2})}\|f\|_{L^{2}(4B)}.\nonumber
\end{eqnarray}
When $j\geq2$, it follows from the $L^{2}-L^{p}$ off-diagonal estimate for $\sqrt{t}\nabla^{m}e^{-tL}$ that
\begin{eqnarray}
\Big(\frac{1}{|B|}\int_{B}\Big|\nabla^{m}e^{-\ell
r^{2m}L}(\chi_{S_{j}(B)}f)(x)\Big|^{q_{0}}dx\Big)^{\frac{1}{q_{0}}}
&\leq& C|B|^{-\frac{1}{q_{0}}}r^{m+N(\frac{1}{q_{0}}-\frac{1}{2})}e^{-2^{\frac{2mj}{2m-1}}}\|f\|_{L^{2}(S_{j}(B))}.\nonumber
\end{eqnarray}

On the other hand, for every $j\geq1$, by Rellich inequality (see \cite[Corollary. 14]{D-H}), we have for $N>2m$
\begin{eqnarray}
\Big(\int_{S_{j}(B)}|f(x)|^{2}dx\Big)^{\frac{1}{2}}&\leq& (2^{j}r)^{m}\Big(\frac{1}{(2^{j}r)^{2m}}\int_{S_{j}(B)}|f(x)|^{2}dx\Big)^{\frac{1}{2}}\nonumber\\
&\leq&C(2^{j}r)^{m}\Big(\int_{\mathbb{R}^{N}}\frac{|\chi^{2}_{S_{j}(B)}f(x)|^{2}}{|x|^{2m}}dx\Big)^{\frac{1}{2}}\nonumber\\
&\leq&C(2^{j}r)^{m}\|\nabla^{m}f\|_{L^{2}(S_{j}(B))}.\nonumber
\end{eqnarray}
Thus we have
\begin{eqnarray}
\Big(\frac{1}{|B|}\int_{B}\Big|\nabla^{m}e^{-\ell
r^{2m}L}f\Big|^{q_{0}}dx\Big)^{\frac{1}{q_{0}}}&\leq& C\sum_{j\geq1}2^{-2mMj+mj+\frac{N}{2}j}
\Big(\frac{1}{|2^{j+1}B|}\int_{2^{j+1}B}|\nabla^{m}f(x)|^{2}dx\Big)^{\frac{1}{2}}\nonumber
\end{eqnarray}
which   means (\ref{eq58}) by choosing large enough M. Now applied to $f=L^{-1/2}g$ gives us (\ref{cond28}).
Therefore, we finish the proof.
\end{proof}

\begin{rem}
{\rm If $V=0$, the restriction $N>2m$ in (ii) of Proposition \ref{pro31} can be removed due to the conversation property $e^{-tL_{0}}1=1$ (see \cite{A2007}). In our paper, we adapt a different approach which require $N>2m$ since $e^{-tL}1=1$ may no longer hold if $V\neq0$.
}
\end{rem}

\subsection{The potential class $L^{\frac{N}{2m},\infty}(\mathbb{R}^{N})$}\label{sec4-2}

In this subsection, we focus on the Riesz transform associated to $L\in \mathscr{A}_{m}(\rho,\Lambda,\mu)$ with extra assumption that $V\in L^{\frac{N}{2m},\infty}(\mathbb{R}^{N})$ on $L^q$ for $q>2$. By Proposition \ref{pro31}, we will mainly consider the $L^{2}-L^{p}$ off-diagonal estimates for $\sqrt{t}\nabla^{m}e^{-tL}$.

First of all, we introduce two important constants.
Denote $s_{m,p}$ ($1\leq p<\frac{N}{m}$) the constant such that
$\|f\|_{L^{p^{\ast}}}\leq s_{m,p}\|\nabla^{m}f\|_{L^{p}}.$
Recall the weak type H\"{o}lder inequality (see \cite[Lemma 4.1]{As}), let $h_{p,q,r}$ be the constant such that
\begin{eqnarray}\label{eq5-1}
\|fg\|_{L^{p}}\leq h_{p,q,r}\|f\|_{L^{r}_{w}}\|g\|_{L^{q}},
\end{eqnarray}
where $f\in L^{r}_{w}(\mathbb{R}^{N})$, $g\in L^{q}(\mathbb{R}^{N})$ and $\frac{1}{p}=\frac{1}{q}+\frac{1}{r}$ for all $p,q,r\in (1,\infty)$.

\begin{rem}\label{rem4-2}{\rm
 Let $N>2m$ and $V\in L^{\frac{N}{2m},\infty}(\mathbb{R}^{N})$,  by the weak H\"{o}lder inequality (\ref{eq5-1}) we have
\begin{eqnarray}\label{extra2}
\int_{\mathbb{R}^{N}}V_{-}(x)|f(x)|^{2}dx&\leq& \|V^{\frac{1}{2}}_{-}f\|_{L^{2}}\|V^{\frac{1}{2}}_{-}f\|_{L^{2}}
\leq h^{2}\|V^{\frac{1}{2}}_{-}\|^{2}_{L^{\frac{N}{m}}_{w}}\|f\|^{2}_{L^{\frac{2N}{N-2m}}}\nonumber\\
&\leq& h^{2}s^{2}_{m,2}\||V|^{\frac{1}{2}}\|^{2}_{L^{\frac{N}{m}}_{w}}\|\nabla^{m}f\|^{2}_{L^{2}}\nonumber\\
&\leq& h^{2}s^{2}_{m,2}\|V\|_{L^{\frac{N}{2m},\infty}}\rho^{-1}
\Big{(}\mathcal {Q}_{0}(f,f)+\int_{\mathbb{R}^N}V_{+}(x)|f(x)|^2dx \Big{)}
\end{eqnarray}
where $\rho$ is defined in (\ref{strong}) and $h=h_{2,\frac{2N}{N-2m},\frac{N}{m}}$ in (\ref{eq5-1}).  Thus if $\mathcal {Q}_{0}$ satisfies (\ref{Q_0 form})-(\ref{strong}) and $V\in L^{\frac{N}{2m},\infty}(\mathbb{R}^{N})$ with
\begin{eqnarray}\label{extra10}
\|V\|_{L^{\frac{N}{2m},\infty}}<\Theta:=\rho(hs_{m,2})^{-2},
\end{eqnarray}
then by (\ref{extra2}), $L=L_{0}+V\in\mathscr{A}_{m}(\rho,\Lambda,\mu)$ automatically where $\mu\leq \Theta^{-1}\|V\|_{L^{\frac{N}{2m},\infty}}$ and $L_{0}$ is the operator associated to $\mathcal {Q}_{0}$.
}
\end{rem}

\begin{lem}\label{le7}
Let $N>2m$ and $L=L_{0}+V\in \mathscr{A}_{m}(\rho,\Lambda,\mu)$ where $L_{0}$ is a homogeneous elliptic operator of order 2m in divergence form defined by (\ref{divergence}) and $V\in L^{\frac{N}{2m},\infty}(\mathbb{R}^{N})$. Then there exist a constant
$r_{0}\in [1,2)$, such that the operator $L+{\bf I}$ is bounded and
invertible from $W^{m,p}(\mathbb{R}^{N})$ to $W^{-m,p}(\mathbb{R}^{N})$ with $p$ satisfying
$|1/2-1/p|<|1/2-1/r_{0}|$.
\end{lem}

\begin{proof}
 Let $f\in W^{m,p}(\mathbb{R}^{N})$ and $g\in
W^{m,p'}(\mathbb{R}^{N})$ with $1/p+1/p'=1$, by the weak H\"{o}lder inequality (\ref{eq5-1}) and Sobolev embedding theorem, we have
\begin{eqnarray}\label{eq02-e}
\int_{\mathbb{R}^{N}}V(x)f(x)\overline{g}(x)dx&\leq&\overline{h}\|V\|_{L^{\frac{N}{2m},\infty}}\|fg\|_{L^{N/N-2m}}\nonumber\\
&\leq&\overline{h}\|V\|_{L^{\frac{N}{2m},\infty}}\|f\|_{L^{Np/N-2m}}\|g\|_{L^{Np'/N-2m}}\nonumber\\
&\leq&\overline{h}s_{m,p}s_{m,p'}\|V\|_{L^{\frac{N}{2m},\infty}}\|f\|_{W^{m,p}}\|g\|_{W^{m,p'}},
\end{eqnarray}
where $\overline{h}=h_{1,\frac{N}{N-2m},\frac{N}{2m}}$ in (\ref{eq5-1}). Thus, since $a_{\alpha,\beta}\in L^{\infty}(\mathbb{R}^{N},\mathbb{C})$, we can obtain
\begin{eqnarray}\label{eq4-4}
&&|\langle (L+{\bf I})f,g\rangle|\nonumber\\
&\le&
\int_{\mathbb{R}^{N}}\bigg(\sum_{|\alpha|=|\beta|=m}\big|a_{\alpha\beta}(x)\partial^{\alpha}f(x)
\overline{\partial^{\beta}g}(x)\big|+\big|V(x)f(x)\overline{g}(x)\big|+\big|f(x)\overline{g}(x)\big|\bigg)dx\nonumber\\
&\leq&C\|f\|_{W^{m,p}}\|g\|_{W^{m,p'}},
\end{eqnarray}
which means that $L+{\bf I}$ is bounded from $W^{m,p}(\mathbb{R}^{N})$ to $W^{-m,p}(\mathbb{R}^{N})$
for all $1<p<\infty$. On the other hand, it follows from the same argument as in the proof of \cite[Proposition 2]{A-T} that $L+{\bf I}$ is invertible from
$W^{m,2}(\mathbb{R}^{N})$ to $W^{-m,2}(\mathbb{R}^{N})$. Thus, applying a result of Sneiberg
\cite{Sn} (see also \cite[~Lemma 23]{A-T}), there exists an
$r_{0}\in [1,2)$, such that the operator $L+{\bf I}$ is bounded and
invertible from $W^{m,p}(\mathbb{R}^{N})$ to $W^{-m,p}(\mathbb{R}^{N})$ with $p$ satisfying
$|1/2-1/p|<|1/2-1/r_{0}|$.
\end{proof}

\begin{rem}
{\rm We need to emphasize that the constant $r_{0}$ may depend on $\rho,\ \Lambda,\ \mu$, $\|V\|_{L^{\frac{N}{2m},\infty}}$ and the Sobolev constants, which will plays an important role in the study of this subsection.
}
\end{rem}

\begin{lem}\label{le8}
Let $N>2m$ and $L=L_{0}+V\in \mathscr{A}_{m}(\rho,\Lambda,\mu)$ where $L_{0}$ is a homogeneous elliptic operator of order 2m in divergence form defined by (\ref{divergence}) and $V\in L^{\frac{N}{2m},\infty}$. Assume that $r_{0}$ is defined as in Lemma \ref{le7}. Then for all $p\in (\frac{Nr_{0}}{N+mr_{0}},2]$, $e^{-tL}$ satisfies the $L^{p}-L^{2}$ estimates.
\end{lem}

\begin{proof}
Assume first $t=1$, it follows from the same method as in the proof of \cite[Corollary 3.1]{D-D-Y1} that $\|e^{-L}\|_{L^{p}-L^{2}}\leq C$ where $p\in (\frac{Nr_{0}}{N+mr_{0}},2]$.

Now let $L_{\delta}=\delta^{-2m}\mathcal {U}^{-1}_{\delta}L\mathcal {U}_{\delta}$ be defined by (\ref{scaling-operator}), it follows that $L_{\delta}\in \mathscr{A}_{m}(\rho,\Lambda,\mu)$ for the same constants $\rho,\Lambda,\mu$ as $L$.
Moreover, for all $\delta>0$ and $V_{\delta}$ defined as in (\ref{extra2-1}), we have that the inequality (\ref{eq02-e}) and (\ref{eq4-4}) also  for $V_{\delta}$ and $L_{\delta}$ respectively with the same bounds.
Thus $L_{\delta}$ and $L$ are the same type which means that $\|e^{-L_{\delta}}\|_{L^{p}-L^{2}}\leq C$ where $p\in (\frac{Nr_{0}}{N+mr_{0}},2]$ and $C>0$ doesnot depend on $\delta$. Then by a scaling argument as in the proof of \cite[Corollary 3.1]{D-D-Y1} or \cite[Corollary 4.5]{A2007}, we can finish the proof.
\end{proof}

\begin{thm}\label{pro5-1}
Let $N>2m$ and $L=L_{0}+V\in \mathscr{A}_{m}(\rho,\Lambda,\mu)$ where $L_{0}$ is a homogeneous elliptic operator of order 2m in divergence form defined by (\ref{divergence}) and $V\in L^{\frac{N}{2m},\infty}(\mathbb{R}^{N})$. There exist $\varepsilon>0$ such that for all $p\in (2,2+\varepsilon)$, $\sqrt{t}\nabla^{m}e^{-tL}$ satisfies the $L^{2}-L^{p}$ estimates.
\end{thm}

\begin{proof}Let $r_{0}$ be defined as in Lemma \ref{le7} and $p>2$ such that $|1/2-1/p|<|1/2-1/r_{0}|$. It is easy to see that $2<p<r'_{0}$. Now let $\varepsilon=r'_{0}-2$, we will prove that $\{\sqrt{t}\nabla^{m}e^{-tL}\}_{t>0}$ satisfy the $L^{2}-L^{p}$ estimate for all $p\in (2,2+\varepsilon)$. To this end, write
$$\nabla^{m}e^{-L}f=\nabla^{m}(L+{\bf I})^{k}(L+{\bf I})^{-k}e^{-L}f,$$
where $k$ will be chosen later. By applying the same procedure as in the proof of  \cite[Corollary 3.1]{D-D-Y1}, we can determine a $k$ such that
$(L+{\bf I})^{-k}$ maps $L^{2}(\mathbb{R}^{N})$ into $W^{m,p}(\mathbb{R}^{N})$, then by the analyticity of
the semigroup $e^{-tL}$, we have that $\|\nabla^{m}e^{-L}\|_{L^{2}-L^{p}}\leq C$ for all $p\in (2,2+\varepsilon)$. Thus, by using scaling as above, we can obtain that $\sqrt{t}\nabla^{m}e^{-tL}$ satisfies the $L^{2}-L^{p}$ estimates for all $p\in (2,2+\varepsilon)$.
\end{proof}
\vskip0.2cm

{\bf \emph{The proof of Theorem \ref{thm5}:}}\ Let $L$ be defined as in Theorem \ref{thm5}, it follows from Theorem \ref{pro5-1} and Proposition \ref{le52}  that $\sqrt{t}\nabla^{m}e^{-tL}$ satisfies the $L^{2}-L^{p}$ off-diagonal estimates for all $p\in (2,2+\varepsilon)$. Thus by using (ii) of Proposition \ref{pro31}, we have that $\nabla^{m}L^{-1/2}$ is bounded on $L^{q}$ for all $q\in (2,2+\varepsilon)$.

On the other hand, by Theorem \ref{pro5-1}, Proposition \ref{le52}, lemma \ref{le005} and duality, we have  that $e^{-tL}$  is bounded on $L^{p}$  for all $p\in(\frac{(2+\varepsilon)N}{(1+\varepsilon)N+(2+\varepsilon)m},\frac{(2+\varepsilon)N}{N-(2+\varepsilon)m})$. Then it follows from Proposition \ref{pro1'}  and (i) of Proposition \ref{prop-2}
that  $\nabla^{m}L^{-1/2}$ is bounded on $L^{q}$ for all $q\in(\frac{(2+\varepsilon)N}{(1+\varepsilon)N+(2+\varepsilon)m},2)$. Hence we finish the proof.

\subsection{Small potentials}\label{sec4-3}

In Theorem \ref{thm5}, we extend the upper bound $2$ to $2+\varepsilon$ by assuming that $L=L_{0}+V\in \mathscr{A}_{m}(\rho,\Lambda,\mu)$ and $V\in L^{\frac{N}{2m},\infty}(\mathbb{R}^{N})$, where $\varepsilon>0$ is small and may depends on $\rho,\ \Lambda,\ \mu$ and $\|V\|_{L^{\frac{N}{2m},\infty}}$. In this subsection, by imposing extra conditions,
we will obtain a larger interval for the boundedness of Riesz transform. First of all, let us see the following theorem dealing with  $\{\sqrt{t}\nabla^{m}e^{-tL}\}_{t>0}$ on $L^{p}$ for some $p>2$.

\begin{thm}\label{ethm2}Let $N>2m$ and $L=L_{0}+V\in \mathscr{A}_{m}$ where $L_{0}$ is a homogeneous elliptic operator of order 2m in divergence form defined by (\ref{divergence}). Assume that  ($A_{1}$) and ($A_{2}$) hold for some $2<q_{0}<\frac{N}{m}$. Then there exist small constant $\delta_{q_{0}}>0$ depending on $q_{0}$ such that when
\begin{eqnarray}\label{eq1-1}
\|V\|_{L^{\frac{N}{2m},\infty}}\leq \delta_{q_{0}},
\end{eqnarray}
$\sqrt{t}\nabla^{m}e^{-tL}$ is  bounded on $L^{q_{0}}(\mathbb{R}^{N})$.
\end{thm}

Before proving Theorem \ref{ethm2}, we consider functional calculus for $L_{0}$ defined by (\ref{divergence}). It is easy to see that the self-adjoint operator $L_{0}$ is densely defined with dense range.  Thus it follows from McIntosh and Yagi \cite{M-Y} that $L_{0}$ has a bounded $H^{\infty}(S^{o}_{\nu})$ functional calculus for all $0<\nu<\pi$ on $L^{2}(\mathbb{R}^{N})$.  Moreover, we have the following result.

\begin{pro}\label{thm-H}
Assume that $L_{0}$ is a homogeneous elliptic operator of order $2m$ in divergence form defined by (\ref{divergence}) and satisfies ($A_{1}$) for some $q_{0}>2$. Then $L_{0}$ has a bounded $H^{\infty}(S^{o}_{\nu})$ functional calculus on $L^{p}(\mathbb{R}^{N})$ for all $\nu>0$ and $p\in (\frac{Nq'_{0}}{N+mq'_{0}},\frac{Nq_{0}}{N-mq_{0}})$ if $q_{0}<\frac{N}{m}$ and $p\in (1,\infty)$ if $q_{0}\geq\frac{N}{m}$.
\end{pro}

\begin{proof}
By hypothesises and (i) of Proposition \ref{pro31} (with $V=0$), it is easy to see that $\sqrt{t}\nabla^{m}e^{-tL_{0}}$ satisfies $L^{2}-L^{q_{0}}$ estimates, which combined Proposition \ref{le52}, Lemma \ref{le005} and the Sobolev embedding theorem imply that
that $e^{-tL_{0}}$ satisfies the $L^{2}-L^{p}$ off-diagonal estimates for $p\in (2,\frac{Nq_{0}}{N-mq_{0}})$ if $q_{0}<\frac{N}{m}$ and $p\in (2,\infty)$ if $q_{0}\geq\frac{N}{m}$.
The $L^{p}-L^{2}$ off-diagonal estimates for $e^{-tL_{0}}$ with corresponding $p$ can be obtained by duality. Then we can finish the proof  by using  Blunck and  Kunstmann \cite[Theorem 1.2]{B-K03}.
\end{proof}

\vskip0.4cm
{\bf \emph{The proof of Theorem \ref{ethm2}:}}
Let us first show that the operators $(\textbf{I}+tL_{0})^{-1/2}|V|^{1/2}$,
$(t^{-1}+L_{0})^{-1/2}|V|^{1/2}$ and $|V|^{1/2}(t^{-1}+L_{0})^{-1/2}$ are bounded on $L^{q_{0}}$. To this end, notice that $(A_{1})$ holds for $2<q_{0}<\frac{N}{m}$, then for all $t>0$, let $F_{t}(z)=(\frac{tz}{1+tz})^{1/2}$ with $\Re z\geq0$ and
$\nu=\frac{\pi}{2}$, by
 Proposition \ref{thm-H}, we have that $F_{t}\in H^{\infty}(\Sigma_{\nu})$ and
\begin{eqnarray}\label{eq6'}
\|L^{1/2}_{0}(t^{-1}+L_{0})^{-1/2}\|_{L^{q_{0}'}-L^{q_{0}'}}\leq c_{q_{0}'},
\end{eqnarray}
where $\frac{1}{q_{0}}+\frac{1}{q_{0}'}=1$. Moreover, it follows from (i) of Remark \ref{rem1-3} that $\nabla^{m}L^{-1/2}_{0}$ is bounded on $L^{q'_{0}}$.
Thus by (\ref{eq6'}) and the weak type H\"{o}lder inequality (\ref{eq5-1}), we have that
\begin{eqnarray}\label{eq6-3}
\Big\||V|^{1/2}(\textbf{I}+tL_{0})^{-1/2}\Big\|_{L^{q_{0}'}-L^{q_{0}'}}&\leq& \Big\||V|^{1/2}L_{0}^{-1/2}\Big\|_{L^{q_{0}'}-L^{q_{0}'}}
\Big\|L_{0}^{1/2}(\textbf{I}+tL_{0})^{-1/2}\Big\|_{L^{q_{0}'}-L^{q_{0}'}}\nonumber\\
&\leq&c_{q_{0}'}t^{-\frac{1}{2}}\Big\||V|^{1/2}L_{0}^{-1/2}\Big\|_{L^{q_{0}'}-L^{q_{0}'}}\nonumber\\
&\leq& c_{q_{0}'}h_{q_{0}'}t^{-\frac{1}{2}}\Big\||V|^{1/2}\Big\|_{L^{\frac{N}{m},\infty}}\Big\|L_{0}^{-1/2}\Big\|_{L^{q_{0}'}-L^{(q_{0}')^{\ast}}}\nonumber\\
&\leq&c_{q_{0}'}h_{q_{0}'}s_{m,q_{0}'}\delta^{\frac{1}{2}}_{q_{0}}t^{-\frac{1}{2}}
\Big\|\nabla^{m}L_{0}^{-1/2}\Big\|_{L^{q_{0}'}-L^{q_{0}'}}
\leq C_{q_{0}'}\delta^{\frac{1}{2}}_{q_{0}} t^{-\frac{1}{2}},
\end{eqnarray}
where  $z_{q'_{0}}:=\|\nabla^{m} L^{-\frac{1}{2}}_{0}\|_{L^{q'_{0}}-L^{q'_{0}}}$, $h_{q_{0}'}=h_{q_{0}',\frac{Nq_{0}'}{N-mq_{0}'},\frac{N}{m}}$ and $C_{q_{0}'}=c_{q_{0}'}h_{q_{0}'}s_{m,q_{0}'}z_{q_{0}'}$. The same procedure as above can be to apply to obtain
\begin{eqnarray}
\Big\||V|^{1/2}(t^{-1}+L_{0})^{-1/2}\Big\|_{L^{q_{0}'}-L^{q_{0}'}}\leq t^{\frac{1}{2}}
\Big\||V|^{1/2}L_{0}^{-1/2}\Big\|_{L^{q_{0}'}-L^{q_{0}'}}
\Big\|L_{0}^{1/2}(\textbf{I}+tL_{0})^{-1/2}\Big\|_{L^{q_{0}'}-L^{q_{0}'}}
\leq C_{q_{0}'}\delta^{\frac{1}{2}}_{q_{0}}.\nonumber
\end{eqnarray}
Then by duality, we have
\begin{eqnarray}\label{eq10'}
\Big\|(\textbf{I}+tL_{0})^{-1/2}|V|^{1/2}\Big\|_{L^{q_{0}}-L^{q_{0}}}
=\Big\||V|^{1/2}(\textbf{I}+tL_{0})^{-1/2}\Big\|_{L^{q_{0}'}-L^{q_{0}'}}
\leq C_{q_{0}'}\delta^{\frac{1}{2}}_{q_{0}}t^{-\frac{1}{2}},
\end{eqnarray}
and
\begin{eqnarray}\label{eq11'}
\Big\|(t^{-1}+L_{0})^{-1/2}|V|^{1/2}\Big\|_{L^{q_{0}}-L^{q_{0}}}
=\Big\||V|^{1/2}(t^{-1}+L_{0})^{-1/2}\Big\|_{L^{q_{0}'}-L^{q_{0}'}}
\leq C_{q_{0}'}\delta^{\frac{1}{2}}_{q_{0}}.
\end{eqnarray}
Similarly to (\ref{eq6-3}),
\begin{eqnarray}\label{eq101'}
\Big\||V|^{1/2}(t^{-1}+L_{0})^{-1/2}\Big\|_{L^{q_{0}}-L^{q_{0}}}\leq C_{q_{0}}\delta^{\frac{1}{2}}_{q_{0}},
\end{eqnarray}
where $z_{q_{0}}:=\|\nabla^{m}L^{-1/2}_{0}\|_{L^{q_{0}}-L^{q_{0}}}$, $h_{p}=h_{q_{0},\frac{Nq_{0}}{N-mq_{0}},\frac{N}{m}}$ and $C_{q_{0}}=c_{q_{0}}h_{q_{0}}s_{m,q_{0}}z_{q_{0}}$.

We denote  $A=\textbf{I}-V(t^{-1}+L_{0})^{-1}$, it is easy to see that the operator $A$ is well defined on $L^{q_{0}}$ by condition ($A_{2}$). Let $$I_{n}=(\textbf{I}+tL_{0})^{-1/2}(V(t^{-1}+L_{0})^{-1})^{n},$$ and $J_{n}=I_{n}A$. Notice that for each $n$, we write
\begin{eqnarray}
I_{n}&=&(\textbf{I}+tL_{0})^{-1/2}\Big(V(t^{-1}+L_{0})^{-1}\Big)\cdots\Big(V(t^{-1}+L_{0})^{-1}\Big)\nonumber\\
&=&(\textbf{I}+tL_{0})^{-1/2}|V|^{1/2} {\rm sign} V|V|^{1/2}(t^{-1}+L_{0})^{-1/2}\nonumber\\
&&\ \ \ \ \Big((t^{-1}+L_{0})^{-1/2}|V|^{1/2} {\rm sign} V|V|^{1/2}(t^{-1}+L_{0})^{-1/2}\Big)\cdots\nonumber\\
&&\ \ \ \ \ \ \Big((t^{-1}+L_{0})^{-1/2}|V|^{1/2} {\rm sign} V|V|^{1/2}(t^{-1}+L_{0})^{-1/2}\Big)(t^{-1}+L_{0})^{-1/2}.\nonumber
\end{eqnarray}
Thus we have for $n\geq1$,
\begin{eqnarray}\label{eq9'}
\|I_{n}\|_{L^{q_{0}}-L^{q_{0}}}
&\leq&\Big\|(\textbf{I}+tL_{0})^{-1/2}|V|^{1/2}
\Big\|_{L^{q_{0}}-L^{q_{0}}}\Big\||V|^{1/2}(t^{-1}+L_{0})^{-1/2}\Big\|^{n}_{L^{q_{0}}-L^{q_{0}}}\nonumber\\
&&\ \ \ \ \ \ \Big\|(t^{-1}+L_{0})^{-1/2}|V|^{1/2}\Big\|^{n-1}_{L^{q_{0}}-L^{q_{0}}}\Big\|(t^{-1}+L_{0})^{-1/2}\Big\|_{L^{q_{0}}-L^{q_{0}}}
\end{eqnarray}
Similarly to (\ref{eq6'}),
\begin{eqnarray}\label{eq13'}
\|(t^{-1}+L_{0})^{-1/2}\|_{L^{q_{0}}-L^{q_{0}}}\leq c_{q_{0}}t^{\frac{1}{2}}.
\end{eqnarray}
Then it follows from (\ref{eq10'})-(\ref{eq13'}) that
\begin{eqnarray}\label{eq14'}
\|I_{n}\|_{L^{p}-L^{p}}\leq c_{p}(C_{p}C_{p'}\delta_{p})^{n},
\end{eqnarray}
which means that $\sum^{\ell}_{n=0}I_{n}$ converges to an operator $T$ on $L^{q_{0}}$ if we choose  $\delta_{q_{0}}<(C_{q_{0}}C_{q_{0}'})^{-1}$. That is, $T=\sum^{\infty}_{n=0}I_{n}$ in the sense of $L^{q_{0}}$. Thus
\begin{eqnarray}
\Big\|TAf-\sum^{\ell}_{n=0}J_{n}f\Big\|_{L^{q_{0}}}
=\Big\|(T-\sum^{\ell}_{n=0}I_{n})Af\Big\|_{L^{q_{0}}}
\leq\Big\|T-\sum^{\ell}_{n=0}I_{n}\Big\|_{L^{q_{0}}-L^{q_{0}}}\|Af\|_{L^{q_{0}}}.\nonumber
\end{eqnarray}
hold for all $f\in L^{q_{0}}$, which implies that
$TAf=\lim_{\ell\rightarrow\infty}\sum^{\ell}_{n=0}J_{n}f$.
Now we claim that
\begin{eqnarray}\label{claim2-1-1'}
\sum^{\ell}_{n=0}(J_{n}f)\rightarrow (\textbf{I}+tL_{0})^{-1/2}f,  \ \ \ \ell\rightarrow\infty
\end{eqnarray}
in the sense of $L^{q_{0}}$. Once the claim (\ref{claim2-1-1'}) holds, we can write
\begin{eqnarray}\label{eq5'}
\nabla^{m}(\textbf{I}+tL)^{-1}&=&\nabla^{m}(\textbf{I}+tL_{0}+tV)^{-1}\nonumber\\
&=&\nabla^{m}(\textbf{I}+tL_{0})^{-1}(\textbf{I}-V(t^{-1}+L_{0})^{-1})^{-1}\nonumber\\
&=&\nabla^{m}(\textbf{I}+tL_{0})^{-1/2}(\textbf{I}+tL_{0})^{-1/2}(\textbf{I}-V(t^{-1}+L_{0})^{-1})^{-1}\nonumber\\
&=&\nabla(\textbf{I}+tL_{0})^{-1/2}\sum^{\infty}_{n=0}I_{n}.
\end{eqnarray}
It follows from  (\ref{eq6'}) and the condition ($A_{1}$) that for all $q\in (1,\infty)$
\begin{eqnarray}\label{eq7'}
\Big\|\nabla^{m}(\textbf{I}+tL_{0})^{-1/2}\Big\|_{L^{q_{0}}-L^{q_{0}}}
&\leq&
\Big\|\nabla^{m}L_{0}^{-1/2}\Big\|_{L^{q_{0}q}-L^{q_{0}}}\Big\|L_{0}^{1/2}(\textbf{I}+tL_{0})^{-1/2}\Big\|_{L^{q_{0}}-L^{q_{0}}}\nonumber\\
&\leq& c_{q_{0}}z_{q_{0}}t^{-\frac{1}{2}}
\end{eqnarray}
Therefore, by (\ref{eq5'}), (\ref{eq7'}) and choosing $\delta_{q_{0}}<(C_{q_{0}}C_{q_{0}'})^{-1}$ in (\ref{eq14'}), we have
\begin{eqnarray}\label{eq6-4}
\|\nabla^{m}(\textbf{I}+tL)^{-1}\|_{L^{q_{0}}-L^{q_{0}}}\leq Ct^{-\frac{1}{2}}.
\end{eqnarray}

If $q_{0}\leq \frac{2N}{N-2m}$, it follows from (\ref{eq6-4}) and the fact that $(tL)^{k}e^{-tL}$ ($k\in \mathbb{N}_{0}$) is bounded on $L^{q_{0}}$ that
\begin{eqnarray}\label{eq15''}
\|\nabla^{m} e^{-tL}\|_{L^{q_{0}}-L^{q_{0}}}\leq \|\nabla^{m} (\textbf{I}+tL)^{-1}\|_{L^{q_{0}}-L^{q_{0}}}\|(\textbf{I}+tL)e^{-tL}\|_{L^{q_{0}}-L^{q_{0}}}\leq Ct^{-\frac{1}{2}}.\nonumber
\end{eqnarray}
However, when
$\frac{2N}{N-2m}<p$, we need more sophisticated discussion. First of all, it follows from the fact that $\nabla^{m} L^{-\frac{1}{2}}$ is bounded on $L^{2}$ and the functional calculi of $L$ on $L^{2}$ that
\begin{eqnarray}\label{eq6-5}
\|\nabla^{m} (\textbf{I}+tL)^{-1}\|_{L^{2}-L^{2}}\leq \|\nabla^{m}L^{-\frac{1}{2}}\|_{L^{2}-L^{2}}\|L^{\frac{1}{2}}(\textbf{I}+tL)^{-1}\|_{L^{2}-L^{2}}\leq Ct^{-\frac{1}{2}}.
\end{eqnarray}
By interpolating (\ref{eq6-4})  with (\ref{eq6-5}), we obtain for all $2\leq r\leq q_{0}$
\begin{eqnarray}\label{eq6-6}
\|\nabla^{m} (\textbf{I}+tL)^{-1}\|_{L^{r}-L^{r}}\leq Ct^{-\frac{1}{2}}.
\end{eqnarray}
Notice that $e^{-tL}$ and $tLe^{-tL}$  are bounded on $L^{r}$ for all $2\leq r\leq \frac{2N}{N-2m}$ (see Propositions \ref{pro1'}, \ref{pro22} and Remark \ref{rem4-2}), then for all $2\leq r<\frac{2N}{N-2m}<q_{0}$, we have
\begin{eqnarray}\label{eq6-7}
\|\nabla^{m} e^{-tL}\|_{L^{r}-L^{r}}\leq \|\nabla^{m} (\textbf{I}+tL)^{-1}\|_{L^{r}-L^{r}}\|(\textbf{I}+tL)e^{-tL}\|_{L^{r}-L^{r}}\leq Ct^{-\frac{1}{2}},
\end{eqnarray}
which combined Lemma \ref{le005} that $e^{-tL}$ is bounded on $L^{r}$ for all $2\leq r<(\frac{2N}{N-2m})^{\ast}$ where $r^{\ast}=\frac{Nr}{N-mr}$ for all $1<r<\frac{N}{m}$. Moreover, by Proposition \ref{pro1'} and the identity $tLe^{-tL}=2e^{-\frac{t}{2}L}\frac{t}{2}Le^{-\frac{t}{2}L}$, we have that $tLe^{-tL}$ is bounded on $L^{r}$ for all $2\leq r<(\frac{2N}{N-2m})^{\ast}$.
Therefore, for all $2\leq r<(\frac{2N}{N-2m})^{\ast}$ (we assume that $(\frac{2N}{N-2m})^{\ast}<q_{0}$, otherwise the proof would be finished), it follows from (\ref{eq6-7}) that
\begin{eqnarray}
\|\nabla^{m} e^{-tL}\|_{L^{r}-L^{r}}\leq \|\nabla^{m} (\textbf{I}+tL)^{-1}\|_{L^{r}-L^{r}}\|(\textbf{I}+tL)e^{-tL}\|_{L^{r}-L^{r}}\leq Ct^{-\frac{1}{2}}.\nonumber
\end{eqnarray}
 Now let $r_{0}\in (2,\frac{2N}{N-2m})$ be chosen later and
 $r_{j}=r_{j-1}^{\ast}=\frac{nr_{j-1}}{n-r_{j-1}}$, we can find suitable
 $r_{0}\in (2,\frac{2N}{N-2m})$ and a integer $j_{0}$ such that $q_{0}=r_{j_{0}}<\frac{N}{m}$.
Then by the same procedure as above, we have that $e^{-tL}$ and $tLe^{-tL}$ are bounded on $L^{r}$ for all $2\leq r< q_{0}^{*}$, which combined (\ref{eq6-6}) implies that  $\sqrt{t}\nabla^{m} e^{-tL}$ is bounded on $L^{q_{0}}$.

It remains to prove (\ref{claim2-1-1'}). For every $f\in L^{q_{0}}$,
\begin{eqnarray}
&&\Big\|\sum^{\ell}_{n=0}J_{n}f-(\textbf{I}+tL_{0})^{-1/2}f\Big\|_{L^{q_{0}}}\nonumber\\
&=&\Big\|I_{\ell+1}f\Big\|_{L^{q_{0}}}
\leq\Big\|(\textbf{I}+tL_{0})^{-1/2}|V|^{1/2}\Big\|_{L^{q_{0}}-L^{q_{0}}}\Big\||V|^{1/2}(t^{-1}+L_{0})^{-1/2}\Big\|^{\ell+1}_{L^{q_{0}}-L^{q_{0}}}\nonumber\\
&&\ \ \ \ \ \ \Big\|(t^{-1}+L_{0})^{-1/2}|V|^{1/2}\Big\|^{\ell}_{L^{q_{0}}-L^{q_{0}}}\Big\|(t^{-1}+L_{0})^{-1/2}f\Big\|_{L^{q_{0}}}\nonumber\\
&\leq&C(C_{q_{0}}C_{q_{0}'}\delta_{q_{0}})^{\ell+1}\|f\|_{L^{q_{0}}},\nonumber
\end{eqnarray}
which implies (\ref{claim2-1-1'}) by choosing $\delta_{q_{0}}<(C_{q_{0}}C_{q_{0}'})^{-1}$. Hence we finish the proof.
\vskip0.2cm

\begin{rem}\label{rem-ext-1'}
{\rm It follows from the proof above that the constant $\delta_{q_{0}}$ can be expressed explicitly by
$\delta_{q_{0}}<(C_{q'_{0}}C_{q_{0}})^{-1}$  where
$$C_{q_{0}'}=c_{q_{0}'}h_{q'_{0},\frac{Nq'_{0}}{N-q'_{0}},N}s_{q'_{0}}\alpha_{q'_{0}}\ \ \ \ \text{and}
\ \ \ \ \ \ \ C_{q_{0}}=c_{q_{0}}h_{q_{0},\frac{Nq_{0}}{N-q_{0}},N}s_{q_{0}}\alpha_{q_{0}}.$$
Moreover, $C_{p}$ ($p\in (q'_{0},q_{0})$) can be obtained by interpolating all the constants in $C_{p_{0}'}$ and $C_{p_{0}}$.
}
\end{rem}
\vskip 0.3cm

Now we turn to prove Theorem \ref{ethm2-2-0}. Let $L=L_{0}+V\in\mathscr{A}_{m}$ with extra conditions ($A_{1}$) and ($A_{2}$), if we further assume that the $L^{\frac{N}{2m},\infty}$ norm of  $V$ is small, then the boundedness of $\nabla^{m}L^{-1/2}$ on $L^{q}$ for some $q>2$ can be obtained.

\vskip 0.2cm

{\bf \emph{The proof of Theorem \ref{ethm2-2-0}:}}
Assume that the condition ($A_{1}$) and ($A_{2}$) are satisfied for $q_{0}\in (2,\frac{N}{m})$, it follows from Theorem \ref{ethm2}, Lemma \ref{le005} and duality that $e^{-tL}$ is bounded on $L^{p}$ for all $p\in ((q_{0}^{*})',q_{0}^{*})$ where $q_{0}^{\ast}=\frac{Nq_{0}}{N-mq_{0}}$ and $1/(q_{0}^{*})'+1/q_{0}^{*}=1$. Then by Propositions \ref{pro1'} and \ref{prop-2}, we have that $\nabla^{m}L^{-\frac{1}{2}}$ is bounded on $L^{q}(\mathbb{R}^{N})$ for all $q\in ((q_{0}^{*})',2]$.

On the other hand, we proved in Theorem \ref{ethm2} that $\sqrt{t}\nabla^{m}e^{-tL}$ is bounded on $L^{q_{0}}$. Then it follows from  Proposition \ref{le52} and (ii) of Proposition \ref{pro31} that $\nabla^{m}L^{-\frac{1}{2}}$ is bounded on $L^{q}$ for all $q\in[2,q_{0})$. Hence we finish the proof.

\section{Applications}

In this section,
we mainly concern the Riesz transform associated to operators which belong to a special subclass of $\mathscr{A}_{m}$.
Let $P(D)$ be a homogeneous elliptic operator of order $2m$ with constant real coefficients, in particular, $P(D)=(-\Delta)^{m}$. Denote by  where
$$\mathscr{B}_{m}=\bigcup_{\rho,\Lambda,\mu}\{L=L_{0}+V\in\mathscr{A}_{m}(\rho,\Lambda,\mu);\ L_{0}=P(D)\}:=\bigcup_{\rho,\Lambda,\mu}\mathscr{B}_{m}(\rho,\Lambda,\mu).$$
By applying Theorems \ref{thm4}, \ref{thm5} and \ref{ethm2-2-0},  we have the following results.

\begin{thm}\label{thm-R0}
Let $m\geq2$, then the following results hold.

{\rm (i)}\ Assume that $L\in\mathscr{B}_{m}$, then $\nabla^{m}L^{-1/2}$ is bounded on $L^{q}(\mathbb{R}^{N})$ for all $q\in (\frac{2N}{N+2m}\vee1,2]$.

{\rm (ii)}\ Assume that $L\in\mathscr{B}_{m}(\rho,\Lambda,\mu)$ and $V\in L^{\frac{N}{2m},\infty}$, then $\nabla^{m}L^{-1/2}$ is bounded on $L^{q}(\mathbb{R}^{N})$ for $q\in(\frac{(2+\varepsilon)N}{(1+\varepsilon)N+(2+\varepsilon)m},2+\varepsilon)$ where $\varepsilon$ is a positive constant.

{\rm (iii)}\ Let $N>2m$ and $L=P(D)+V$. Assume that ($A_{2}$) holds for some $q_{0}\in (2,\frac{N}{m})$. Then there exist a constant $\delta_{q_{0}}>0$ depending on $q_{0}$ such that when
\begin{eqnarray}\label{eq1-1-1'}
\|V\|_{L^{\frac{N}{2m},\infty}}\leq \delta_{q_{0}},
\end{eqnarray}
$\nabla^{m}L^{-1/2}$ is bounded on $L^{q}(\mathbb{R}^{N})$ for all $\frac{Nq_{0}'}{N+mq_{0}'}<q<q_{0}$.
\end{thm}

\begin{proof}Since the conclusions (i) and (ii) can be obtained directly by using Theorems  \ref{thm4} and \ref{thm5} respectively,
we only need to prove (iii). In fact,
by  Remark \ref{rem4-2}, it is easy to see that $L\in\mathscr{B}_{m}$ if one choose $\delta_{q_{0}}$ appropriately in (\ref{eq1-1-1'}). On the other hand, it follows from Remark \ref{rem1-3} that if $L_{0}=P(D)$, the condition ($A_{1}$) holds for all $q_{0}\in (2, \infty)$. Then we can finish the proof by applying Theorem  \ref{ethm2-2-0}.
\end{proof}

\begin{rem}
{\rm (i)\  In statement (iii) of Theorem \ref{thm-R0}, if $N>4m$, $L_{0}=P(D)$ and $V\in L^{\frac{N}{2m},\infty}$, the condition ($A_{2}$) holds for all $q_{0}\in (2,\frac{N}{2m})$ automatically (see (ii) of Remark \ref{rem1-3}).

(ii)\ If $m=1$, $N\geq3$ and $L\in\mathscr{B}_{1}(\rho,\Lambda,\mu)$. By using the positivity  and contractivity of $e^{-t(P(D)+V_{+})}$,  Assaad \cite{As} proved that $\nabla L^{-1/2}$ is bounded on $L^q$ for all $(\frac{2N}{N+2+(N-2)\sqrt{1-\mu}},2]$ and the lower bound  is sharp. However, when $N>2m\geq4$, it is still unknown to us  if the interval in (i) of Theorem \ref{thm-R0} is sharp (with respect to $\mathscr{B}_{m}(\rho,\Lambda,\mu)$ or $\mathscr{B}_{m}$).
}
\end{rem}
\vskip0.2cm

At last, we treat the higher order Schr\"{o}dinger operators $L=(-\Delta)^{m}-\gamma|x|^{-2m}$. Let $\rho$ be the constant such that (\ref{strong}) holds with $L_{0}=(-\Delta)^{m}$ and $\kappa(m,N)$ is defined by (\ref{H-R}).

\begin{cor}\label{cor2}
Let $m\geq2$ and $L=(-\Delta)^{m}-\gamma|x|^{-2m}$, the following statements hold.

{\rm(i)}\ Let $N>2m$ and $\gamma<\frac{\rho}{\kappa(m,N)}$. Then the Riesz transform $\nabla^{m}L^{-1/2}$ is bounded on $L^{q}(\mathbb{R}^{N})$ for $q\in(\frac{(2+\varepsilon)N}{(1+\varepsilon)N+(2+\varepsilon)m},2+\varepsilon)$ where  $\varepsilon$ is a positive constant.

{\rm(ii)}\ Let $N>4m$. Then  there exist a  constant $\delta>0$ such that when $\gamma<\delta$, the Riesz transform $\nabla^{m}L^{-1/2}$ is bounded on $L^{q}(\mathbb{R}^{N})$ for all $\frac{N}{N-m}<q<\frac{N}{2m}$.
\end{cor}

\begin{proof} The  conclusion (i) is actually proved in (ii) of Remark \ref{rem1-1},
we only need to prove (ii). To this end, notice that by Davies and Hinz \cite[Corollary 14]{D-H},   there exist a constant $K(m,N,p)$ such that for all $1<p<\frac{N}{2m}$
\begin{eqnarray}\label{eq1-2}
\|f(x)|x|^{-2m}\|_{L^{p}}\leq K(m,N,p)\|(-\Delta)^{m}f\|_{L^{p}},
\end{eqnarray}
which means that the condition  ($A_{2}$) holds for all $2<q_{0}<\frac{N}{2m}$ when $N>4m$. Thus by (iii) of Theorem \ref{thm-R0}, there exist a small constant $\overline{\delta}_{q_{0}}$ such that when
$$\|V\|_{L^{\frac{N}{2m},\infty}}=\gamma \||x|^{-2m}\|_{L^{\frac{N}{2m},\infty}}<\overline{\delta}_{q_{0}},$$
i.e. $\gamma<\delta_{q_{0}}:=\overline{\delta}_{q_{0}} \||x|^{-2m}\|^{-1}_{L^{\frac{N}{2m},\infty}}$, the Riesz transform $\nabla^{m}L^{-1/2}$ is bounded on $L^{q}(\mathbb{R}^{N})$ for all $q\in (\frac{Nq_{0}'}{N+mpq_{0}'},q_{0})$.

It remains to show that $\delta_{q_{0}}$ is uniformly bounded for all $q_{0}\in (2,\frac{N}{2m})$. In fact, it follows from Remark \ref{rem4-2} and the proof of Theorem \ref{ethm2} that
$\overline{\delta}_{q_{0}}<\min\{\Theta, (C_{q_{0}}C_{q_{0}'})^{-1}\}$ where $\Theta$ is defined as (\ref{extra10}),
 $$C_{q_{0}'}=c_{q_{0}'}h_{q'_{0},\frac{Nq'_{0}}{N-q'_{0}},N}s_{q'_{0}}\alpha_{q'_{0}}\ \ \ \ \text{and}
\ \ \ \ \ \ C_{q_{0}}=c_{q_{0}}h_{q_{0},\frac{Nq_{0}}{N-q_{0}},N}s_{q_{0}}\alpha_{q_{0}}.$$
On the other hand, it is easy to see that $C_{2}$ and $C_{\frac{N}{2m}}$ is finite. Then by Remark \ref{rem-ext-1'}, we have that
$\overline{\delta}_{q_{0}}$ is uniformly
bounded for all $2< p_{0}<\frac{N}{2m}$, which finishes the proof.
\end{proof}

\begin{rem}{\rm (i)\ When $m=1$, recently Hassel and Lin \cite{H-L} have obtained the sharp interval for the boundedness of $\nabla L^{-1/2}$ on $L^{q}$ based on a different method. Moreover, Killip et al. \cite{KMVZZ} studied the generalized Riesz transform and obtained
\begin{eqnarray}
\big\|(-\Delta)^{s/2}\big\|_{L^q}\leq C\big\|L^{s/2}f\big\|_{L^{q}}\nonumber
\end{eqnarray}
for some $q$ by using the heat kernel estimates and Littlewood-Paley argument. However,  our corollary can clearly  deal with higher order case.

(ii)\  When $m=2$, that is, $L=\Delta^{2}-\gamma|x|^{-4}$, Gregorio \cite{Gre} obtained the $L^{q}$ bounededness of $\Delta L^{-1/2}$ for $q\leq 2$.

}
\end{rem}

\vspace{1cm}
{\bf Qingquan Deng}\\
Department of Mathematics \\
Hubei Key Laboratory of Mathematical Science\\
Central China Normal University\\
Wuhan 430079, China \\
E-mail: {\it dengq@mail.ccnu.edu.cn} \vskip 0.5cm

\vskip 0.2cm
{\bf Yong Ding}\\
School of Mathematical Sciences\\
Laboratory of Mathematics and Complex Systems (BNU)\\
Ministry
of Education 
Beijing Normal University\\
Beijing 100875, China \\
E-mail: {\it dingy@bnu.edu.cn} 
\vskip 0.2cm 

{\bf Xiaohua Yao}\\
Department of Mathematics \\
Hubei Key Laboratory of Mathematical Science\\
Central China Normal University\\
Wuhan 430079, China \\
E-mail: {\it yaoxiaohua@mail.ccnu.edu.cn}


\bigskip
\begin{thebibliography}{10}

\bibitem {As} J. Assaad, \emph{Riesz transform associated to Schr\"{o}dinger operators with  negative potentials,} Publ. Mat., \textbf{55} (2011), 123-150.

\bibitem{A-O} J. Assaad and E. M. Ouhabaz, \emph{Riesz Transform of Schr\"{o}dinger Operators on Manifolds}, J. Geom. Anal., \textbf{22} (2012), 1108-1136.


\bibitem {A2004}P. Auscher, \emph{On $L^p$-estimates
for square roots of second order
elliptic operators on $\mathbb{R}^n$}, Publ. Mat., \textbf{48}
(2004), 159-186.

\bibitem {A2007} P. Auscher, \emph{On necessary and sufficient conditions
for $L^{p}$ estimates of Riesz transforms associated to elliptic
operators on $\mathbb{R}^{n}$ and related estimates,} Mem.
Amer. Math. Soc., \textbf{186}, no. 871 (2007).

\bibitem{A-A} P. Auscher and B. Ben Ali, \emph{Maximal inequalities and Riesz transform estimates on $L^p$ spaces for schr\"{o}dinger operators
with nonnegative potentials,} Ann. Inst. Fourier,  \textbf{57} (2007) 1975-2013.

\bibitem{A-C}P. Auscher and T. Coulhon, \emph{Riesz transforms on manifolds and Poincar\'{e} inequalities}, Ann. Sc. Norm.
Super. Pisa, Cl. Sci., \textbf{4} (2005), 1-25.



\bibitem {A-C-D-H} P. Auscher, T. Coulhon, X. T. Duong and  S. Hofmann,
\emph {Riesz transform on mainfords and heat kernel regularity},
Ann. Sci. \'{E}cole Norm. Sup., \textbf{37} (2004), 911-957.

\bibitem{A-H-M-T}P. Auscher, S. Hofmann, A. McIntosh and P. Tchamitchian, \emph{The Kato square root problem for higher order elliptic
operators and systems on $\mathbb{R}^{n}$}, J. Evol. Equ., \textbf{1} (2001) 361-385.

\bibitem{ASLMT} P. Auscher, S. Hofmann, M. Lacy, A. McIntosh and P. Tchamitchian,
\emph{The solution of the Kato square root problem for second order elliptic operators on ${\Bbb R}^n$,}
Ann of Math., \textbf{156} (2002),  633-654.

\bibitem {A-T}P. Auscher and P. Tchamitchian,
\emph{Square root problem for divergence operators and related
topics}, Asterisque, \textbf{249}, Soc. Math. France, 1998.

\bibitem{A-Q} P. Auscher and M. Qafsaoui, \emph{Equivalence between Regularity Theorems and Heat
Kernel Estimates for Higher Order Elliptic Operators and Systems
under Divergence Form}, J. Funct. Anal., \textbf{177} (2000),
310-364.



\bibitem {B-D}G. Barbatis and E. B. Davies,
\emph{Sharp bounds on heat kernels of higher order uniformly
elliptic operators}, J. Oper. Theory, \textbf{ 36} (1996), 179-198.

\bibitem {M-B}M.  Beceanu, \emph{New estimates for a time-dependent Schr\"{o}dinger equation},  Duke Math. J., \textbf{159} (2011),  417-477.


\bibitem {B-K03}S. Blunck and P. C.  Kunstmann,  \emph{Calder\'{o}n-Zygmund theory for nonintegral
operators and the $H^\infty$-functional calculus}, Rev. Mat.
Iberoamericana, \textbf{19} (2003), 919-942.


\bibitem {B-K04-1}S. Blunck and P. C. Kunstmann,  \emph{Weak type (p,p) estimates for Riesz transforms,} Math. Z., \textbf{247} (2004), 137-148.


\bibitem{B-P-S-T} N. Burq, F. Planchon, J.G. Stalker and A.S. Tahvildar-Zadeh, \emph{Strichartz estimates for the wave and
Schr\"{o}dinger equations with the inverse-square potential}, J. Funct. Anal., \textbf{203} (2003), 519-549.

\bibitem{C-D07} T. Coulhon and N. Dungey, \emph{Riesz transform and pertubation,} J. Geom. Anal., \textbf{17} (2007), 213-226.

\bibitem{C-D99} T. Coulhon and X. Duong, \emph{Riesz transform for $1\leq p\leq 2$,} Trans. Amer. Math. Soc., \textbf{35} (1999), 1151-1169.

\bibitem{C-D01} T. Coulhon and X. Duong, \emph{Riesz transform for $p>2$,} C. R. A. S. Paris, \textbf{332}, 11, s\'{e}rie I, (2001), 975-980.


\bibitem {D95-1}E. Davies, \emph{Long time asympototics of the fourth order parabolic equations,} J. Anal. Math., \textbf{67}
(1995), 323-345.

\bibitem {D95-3}E. Davies, \emph{Uniformly elliptic operators with
mesrsurable cosfficients,} J. Funct. Anal., \textbf{132} (1995),
141-169.

\bibitem {D97-1}E. Davies, \emph{Limits on $L^p$ regularity of
self-adjoint elliptic operators,} J. Diff. Equ., \textbf{135}
(1997), 83-102.



\bibitem{D-H}E. Davies and A. Hinz, \emph{Explicit constants for Rellich inequality in $L_{p}(\Omega)$},
Math. Z., \textbf{227} (1998), 511-523.



\bibitem {D-D-Y1}Q. Deng, Y. Ding and X. Yao, \emph{Characterizations of Hardy spaces associated to higher order
elliptic operators, } J. Funct. Anal., \textbf{263} (2012), 604-674.

\bibitem {D-D-Y2}Q. Deng, Y. Ding and X. Yao, \emph{$L^{q}$ estimates of  Riesz transforms associated to Schr\"{o}dinger operators,} J. Aust. Math. Soc., to appear.

\bibitem{D-D-Y4}Q. Deng, Y. Ding and X. Yao, \emph{Gaussian bounds for higher-order elliptic differential operators
with Kato type potentials}, J. Funct. Anal., \textbf{266} (2014) 5377-5397.





\bibitem {D-M}X. Duong and A. Mcintosh, \emph{The $L^{p}$-boundedness of Riesz transforms associated with divergence form operators,} Proceeding  of
the Centre for Mathematical Analysis, ANU, Canberra, \textbf{37} (1999), 15-25.


\bibitem{D-O-Y}X. Duong, E. M. Ouhabaz and L. Yan, \emph{Endpoint estimates for Riesz transform of magnetic Schr\"{o}dinger operators,} Ark. Mat., \textbf{44} (2006), 261-275.

\bibitem {L-E}L. Evans, \emph{Partial differential equations}, Graduate Studies in Mathematics, \textbf{19} American Mathematical Society.


\bibitem{F}C. Fefferman, \emph{The uncertainty principle}, Bull.  Amer. Math. Soc., N.S. \textbf{9} (1983), 129-206.

\bibitem{Gre}F. Gregorio, \emph{Fourth-order Schr\"{o}dinger type operator with singular potentials}, arXiv: 1601.05243v1.


\bibitem{M-G}M. Goldberg, \emph{Dispersive bounds for the three-dimensional Schr\"{o}dinger equation with almost critical potentials}, Geom. Funct. Anal., \textbf{16} (2006), 517-536.


\bibitem{H-L}A. Hassell and P. Lin, \emph{The Riesz transform for homogeneous Schr\"{o}dinger operators on metric cones}, Revista Mat. Iberoamericana, to appear.

\bibitem {H-M03}S. Hofmann and J. Martell, \emph{$L^{p}$ bounds for Riesz
transforms and square roots associated to the second order elliptic
operators,} Publ. Mat., \textbf{47} (2003),  497-515.

\bibitem {H-M09}S. Hofmann and S. Mayboroda, \emph{Hardy and
BMO spaces associated to divergence form elliptic operators,} Math.
Ann., \textbf{344} (2009),  37-116.


\bibitem{Ka}T. Kato, \emph{Perturbation Theory for Linear Operators}, 2nd edtion, Springer-Verlag, 1980.

\bibitem{KMVZZ}R. Killip, C. Miao, M. Visan, J. Zhang and J. Zheng, \emph{Multipliers and Riesz transforms for the Schr\"{o}dinger operator with inverse-square potential},
arXiv:1503.02716v1.

\bibitem {L-M} M. Langer and V. Maz'ya, \emph{On $L^p$-Contractivity of Semigroups Generated
by Linear Partial Differential Operators}, Jour. Funct. Anal.,
\textbf{164} (1999) 73-109.

\bibitem{Li99}H. Li, \emph{La transformation de Riesz sur les vari\'{e}t\'{e}s coniques,} J. Funct. Anal., \textbf{168}(1999), 145-238.



\bibitem{L-S-V} V. Liskevich. Z. Sobol and H. Vogt, \emph{On the $L^p$-theory of $C_{0}$-semigroups associated with second-order elliptic operators. II,} J. Funct. Anal., \textbf{193} (2002), 55-76.


\bibitem {Mc} A. McIntosh, \emph{Operators which have an
$H^{\infty}$ calculus,} Miniconference on operator theory and
partial differential equations, Proc. Centre Math. Analysis,
ANU, Canberra, \textbf{14} (1986), 210-231.


\bibitem {M-Y} A. McIntosh and A. Yagi.
\emph{Operators of type $\omega$ without a bounded $H^{\infty}$ functional
calculus},  Miniconference on Operators in Analysis, Proceedings
 of Centre for Mathematical Analysis, ANU, Canberra, \textbf{24}
(1989), 159-174.




\bibitem {Ma} V. Maz'ya, \emph{Sobolev spaces, with the applications to elliptic partial differential equations}, Second edition, Springer, 2011.

\bibitem{Ou}E. M. Ouhabaz, \emph{Analysis of Heat Equations on Domains}, London Math. Soc. Monogr., vol. 31, Princeton Univ.
Press, 2005.




\bibitem{RS2}M. Reed and B. Simon, \emph{Methods of modern mathematical physics. II. Fourier analysis, self-adjointness,} Academic Press, New York, 1975

\bibitem{R-S78}M. Reed and B. Simon, \emph{Methods of modern mathematical physics. IV. Analysis of operators,} Academic Press, New York, 1978.





\bibitem{Sc}M. Schechter, \emph{Spectra of partial differential operators,} 2nd edn Elsevier Science Publishers B.V.,
Amsterdam, 1986.

\bibitem{Sh95}Z. Shen, \emph{$L^p$ estimates for Schr\"{o}dinger operators with certain potentials,} Ann. Inst. Fourier, \textbf{45}(1995), 513-546.


\bibitem{Si} A. Sikora, \emph{Riesz transform, Guassian bounds and the method of wave equation,} Math. Z., \textbf{247} (2004), 643-662.


\bibitem {Si82} B. Simon, \emph{Schr\"{o}dinger semigroups},
Bull. Amer. Math. Soc. N.S. \textbf{7} (1982), 447-526.


\bibitem {Sn}I. Sneiberg, \emph{Spectral properties of linear operators in interpolation families of banach
space}, Mat. Issled, \textbf{9} (1974), 214-229.

\bibitem{St}E. M. Stein, \emph{Harmonic Analysis, Real Variable Methods, Orthogonality, and Oscillatory Integrals,} Princeton Univ. Press, Princeton, NJ, 1993.


\bibitem{G-T}G. Talenti, \emph{Best constant in Sobolev inequality}, Ann. di Matem. Pure ed Appl., \textbf{110} (1976), 353-372.

\bibitem{S-T}S. Thangavelu, \emph{Riesz transform and the wave equation for the Hermite operators}, Comm. in P.D.E.,  \textbf{8} (1990), 1199-1215.


\bibitem {U-Z} R. Urban and J. Zienkiewicz, \emph{Dimension free estimates for Riesz transforms of some Schr\"{o}dinger operators}, Israel  Journal of Mathematics, \textbf{173} (2009), 157-176.

\bibitem {V-Z} J. L. Vazquez and E. Zuazua, \emph{The Hardy inequality and the asymptotic behaviour of the heat equation
with an inverse-square potential}, J. Funct. Anal., \textbf{173} (2000), 103-153.


\bibitem{Zh}J. Zhong, \emph{Harmonic analysis for some Schr\"{o}dinger type operators}, PH.D. Thesis, Princeton University, 1993.




\end{thebibliography}
\end{document}